\documentclass[11pt]{article}
\usepackage[utf8]{inputenc}
\usepackage[T1]{fontenc}
\usepackage[noBBpl]{mathpazo}
\usepackage{microtype}
\usepackage{setspace}
\usepackage{needspace}

\usepackage{tikz}
\usetikzlibrary {decorations.pathmorphing, decorations.pathreplacing, decorations.shapes, patterns, through, math, calc, positioning, fit, shapes.misc, hobby}
\usepackage{caption, subcaption}
\usepackage{tkz-euclide}
\tikzset{snake it/.style={decorate, decoration=snake}}

\usepackage[left = 2.5cm, right = 2.5cm, top = 2.5cm, bottom = 2.5cm, headsep = 12pt]{geometry}
\usepackage{parskip}

\usepackage[dvipsnames]{xcolor}
\definecolor{cblue}{RGB}{0,70,140}
\definecolor{cgreen}{RGB}{100,140,0}
\definecolor{cred}{RGB}{190,10,50}

\usepackage{amsmath}
\usepackage{amsthm}
\usepackage{amssymb}
\usepackage{xcolor}
\usepackage{graphicx}
\usepackage{float}
\usepackage{enumitem}
\usepackage{thmtools}
\usepackage{lipsum}
\usepackage{mathtools}
\usepackage{braket}

\DeclareFontFamily{U} {cmr}{}

\DeclareFontShape{U}{cmr}{m}{n}{
	<-6> cmr5
	<6-7> cmr6
	<7-8> cmr7
	<8-9> cmr8
	<9-10> cmr9
	<10-12> cmr10
	<12-> cmr12}{}

\DeclareSymbolFont{Xcmr} {U} {cmr}{m}{n}
\DeclareSymbolFont{cmletters}{OML}{cmm}{m}{it}
\DeclareMathSymbol{\cmchi}{\mathord}{cmletters}{"1F}

\renewcommand{\chi}{\cmchi}
\newcommand{\Vberg}{V_{RGB}}
\newcommand{\Vr}{V_{RDW}}
\newcommand{\Vb}{V_{BWY}}
\newcommand{\Vg}{V_{GDY}}

\newcommand{\abund}{\bigl\lfloor{(n+1)}/{4} \bigr\rfloor}

\usepackage[hidelinks,colorlinks,pagebackref]{hyperref} % 
\hypersetup{citecolor=green!70!black,linkcolor=blue!70!black,urlcolor=blue!70!black}
\usepackage[capitalise,nameinlink,noabbrev,compress]{cleveref}

\DeclareFontFamily{U}{matha}{\hyphenchar\font45}
\DeclareFontShape{U}{matha}{m}{n}{
	<5> <6> <7> <8> <9> <10> gen * matha
	<10.95> matha10 <12> <14.4> <17.28> <20.74> <24.88> matha12
}{}
\DeclareSymbolFont{matha}{U}{matha}{m}{n}

\DeclareMathSymbol{\specialuparrow}{\mathrel}{matha}{"D2}
\DeclareMathSymbol{\specialrightarrow}{\mathrel}{matha}{"D1}

\renewcommand*{\backref}[1]{}
\renewcommand*{\backrefalt}[4]{
	\ifcase #1 Not cited.%
	\or $\specialuparrow$#2%
	\else $\specialuparrow$#2%
	\fi%
}

\makeatletter
\renewcommand*{\eqref}[1]{%
  \hyperref[{#1}]{\textup{\tagform@{\ref*{#1}}}}%
}
\makeatother

\newcommand{\affiliation}{\footnote}
\newcommand{\affiliationmark}[1][\value{footnote}-1]{\footnotemark[\numexpr#1+1\relax]}
\newcommand{\define}[1]{\textcolor{Maroon}{\emph{#1}}}

\renewcommand{\phi}{\varphi}

\let\emptyset\varnothing

\theoremstyle{plain}
\newtheorem{thm}{Theorem}[section]
\newtheorem{claim}{Claim}[thm]
\newtheorem{conjecture}[thm]{Conjecture}
\newtheorem{lem}[thm]{Lemma}

\newtheorem{obs}[thm]{Observation}

\newtheorem{prob}[thm]{Problem}
\theoremstyle{definition}
\newtheorem{defn}[thm]{Definition}



\makeatletter
\newcommand{\namelabel}[1]{%
  \phantomsection
  \renewcommand{\@currentlabel}{#1}% Update the label text/name
  \label{#1}% Set the label
}
\makeatother

\crefalias{lemmai}{lemma}

\newcommand{\oldqed}{}
\def\endofClaim{\hfill\scalebox{.6}{$\Box$}}
\newenvironment{claimproof}[1][Proof]{
  \renewcommand{\oldqed}{\qedsymbol}
  \renewcommand{\qedsymbol}{\endofClaim}
  \begin{proof}[#1]}
  {\end{proof}
  \renewcommand{\qedsymbol}{\oldqed}} 

\setlist[enumerate]{itemindent=0.5\parindent+\labelwidth, 
leftmargin =*, labelindent = 0.5 \parindent}
\newcommand{\itmarab}[1]{\mbox{\hspace{0.05em}\rm 
({#1}\:\!\arabic{*}\hspace{0.05em})}}

\usepackage{xparse}
\NewDocumentCommand{\drawroundedtriangle}{O{2} O{blue} m m m}{%
  \coordinate (CENTROID) at ($1/3*(#3) + 1/3*(#4) + 1/3*(#5)$);

  \coordinate (Aout) at ($(CENTROID)!#1!(#3)$);
  \coordinate (Bout) at ($(CENTROID)!#1!(#4)$);
  \coordinate (Cout) at ($(CENTROID)!#1!(#5)$);

  \draw[rounded corners=12pt, line width=1pt,
        fill=#2!10, draw=#2!70!black, opacity=0.7]
    (Aout) -- (Bout) -- (Cout) -- cycle;
}

\definecolor{DarkDesaturatedBlue}{HTML}{3A3556}
\definecolor{VividOrange}{HTML}{F15918}
\definecolor{PureOrange}{HTML}{FFBA00}
\definecolor{LightGrayishPink}{HTML}{EEC5D5}
\definecolor{VerySoftBlue}{HTML}{B5AFDB}

\newcommand{\triple}[7]{
	
	\ifx\relax#4\relax
	\def\qoffs{0pt}
	\else
	\def\qoffs{#4}
	\fi
	
	\def\qhedge{
		($#1+#3!\qoffs!-90:#2-#3$) --
		($#2+#1!\qoffs!-90:#3-#1$) --
		($#3+#2!\qoffs!-90:#1-#2$) -- cycle}
	
	\coordinate (12) at ($#1!\qoffs!90:#2$);
	\coordinate (13) at ($#1!\qoffs!-90:#3$);
	\coordinate (23) at ($#2!\qoffs!90:#3$);
	\coordinate (21) at ($#2!\qoffs!-90:#1$);
	\coordinate (31) at ($#3!\qoffs!90:#1$);
	\coordinate (32) at ($#3!\qoffs!-90:#2$);
	
	\def\nqhedge{
		(13) let \p1=($(13)-#1$), \p2=($(12)-#1$) in
		arc[start angle={atan2(\y1,\x1)}, delta angle={atan2(\y2,\x2)-atan2(\y1,\x1)-360*(atan2(\y2,\x2)-atan2(\y1,\x1)>0)}, x radius=\qoffs, y radius=\qoffs] --
		(21) let \p1=($(21)-#2$), \p2=($(23)-#2$) in
		arc[start angle={atan2(\y1,\x1)}, delta angle={atan2(\y2,\x2)-atan2(\y1,\x1)-360*(atan2(\y2,\x2)-atan2(\y1,\x1)>0)}, x radius=\qoffs, y radius=\qoffs] --
		(32) let \p1=($(32)-#3$), \p2=($(31)-#3$) in
		arc[start angle={atan2(\y1,\x1)}, delta angle={atan2(\y2,\x2)-atan2(\y1,\x1)-360*(atan2(\y2,\x2)-atan2(\y1,\x1)>0)}, x radius=\qoffs, y radius=\qoffs] --
		cycle}
	
	\ifx\relax#5\relax
	\def\qlwidth{1pt}
	\else
	\def\qlwidth{#5}
	\fi
	
	\ifx\relax#7\relax
	\fill \nqhedge;
	\else
	\fill[#7]\nqhedge;
	\fi
	
	\ifx\relax#6\relax
	\draw[line width=\qlwidth,rounded corners=\qoffs]\nqhedge;
	\else
	\draw[line width=\qlwidth,#6]\nqhedge;
	\fi
}

\title{Spanning tight components in $4$-uniform hypergraphs}
\author{Francesco Di Braccio\affiliation{Department of Mathematics, London School of Economics and Political Science, London, United Kingdom (\textsf{\href{mailto:f.di-braccio@lse.ac.uk}{f.di-braccio},\href{mailto:b.hearn@lse.ac.uk}{b.hearn},\href{mailto:j.m.lada@.lse.ac.uk}{j.m.lada},\href{m.s.neve@lse.ac.uk}{m.s.neve},\href{l.zhang100@lse.ac.uk}{l.zhang100@lse.ac.uk}}).} \and Brian Hearn\affiliationmark[1] \and Joanna Lada\affiliationmark[1] \and Mihir Neve\affiliationmark[1] \and Lu-Ming Zhang\affiliationmark[1]}

\begin{document}
\maketitle 
\begin{abstract}
    We prove that every $n$-vertex $4$-uniform hypergraph with minimum codegree at least $\lfloor n/4 \rfloor$ has a spanning tight component. This is tight, and it settles the $4$-uniform case of a conjecture of Illingworth, Lang, M\"uyesser, Parczyk, and Sgueglia.
\end{abstract}

\section{Introduction}

Connectivity is a central notion in the study of spanning structures in graphs. Many intensively studied classes of graphs in this setting---such as cycles, trees, powers of cycles---are connected. Consequently, identifying conditions that guarantee global connectivity is both a prerequisite and a natural first step towards embedding such structures. Indeed, a common approach is to decompose the target structure into pieces which are embedded separately, and then use the host graph's robust connectivity properties to link these pieces together. A classical illustration is provided by minimum degree conditions. Dirac's seminal theorem \cite{Dir52} shows that, for an $n$-vertex graph, minimum degree $n/2$ is not only the threshold for connectivity but also forces the existence of a Hamilton cycle. Koml\'os, S\'ark\"ozy, and Szemer\'edi~\cite{kss} proved, via the method described, that asymptotically the same minimum degree suffices to ensure the presence of all trees with maximum degree $O(n/\log n)$. Applications of this method are ubiquitous; to mention a few representative examples, see~\cite{MONTGOMERY2019106793, 1factor, ringel, rodlrucinskiszemeredi}.

In this paper, we turn to the analogous question for hypergraphs: investigating degree conditions that force the emergence of global connectivity. The specific question we will investigate has its roots in a problem of Conlon and, indepedently, Gowers (see \cite{spanningsurfaces}) concerning so-called \emph{topological spheres}. Given a $k$-uniform hypergraph (or $k$-graph) $H$, a \define{copy of a $(k-1)$-sphere} is a subgraph whose down-closure induces a simplicial complex that is homeomorphic to the $(k-1)$-sphere $\mathbb{S}^{k-1}$. A subgraph $H'$ of $H$ is \define{spanning} if every vertex of $H$ lies in an edge of $H'$. 

Conlon and Gowers asked what degree conditions force a $k$-graph $H$ to contain a spanning copy of a $(k-1)$-sphere. A spanning $1$-sphere in a graph is simply a Hamilton cycle, and so this question may be viewed as a topological extension of Dirac's theorem. For higher uniformities, one of the most natural and widely studied degree conditions is the \define{minimum codegree}, defined as the minimum over all $(k-1)$-tuples $S \in \binom{V(H)}{k-1}$ of the number of edges containing $S$. In this vein, Georgakopoulos, Haslegrave, Montgomery, and Narayanan \cite{spanningsurfaces} proved that any $n$-vertex $3$-graph with minimum codegree at least $n/3 + o(n)$ contains a spanning copy of a $2$-sphere, and further conjectured that every $n$-vertex $k$-graph $H$ with minimum codegree at least $n/k$ contains a spanning copy of a $(k-1)$-sphere. They also provided a construction showing that this result and conjecture are tight up to an additive term of $k-1$.

Interestingly, $(k-1)$-spheres in $k$-graphs exhibit a strong form of connectivity captured by the following notion. A $k$-graph $H$ is said to be \define{tightly connected} if, for any $e, e' \in E(H)$, there exists a sequence of edges $f_1, \dots, f_\ell \in E(H)$ such that $e = f_0$, $e' = f_\ell$, and $|f_i \cap f_{i+1}|= k-1$ for all $i \in [\ell-1]$. A \define{tight component} is an edge-maximal tightly connected subgraph. In \cite{spanningsurfaces}, the authors construct $k$-graphs with minimum codegree roughly $n/k$ that not only avoid every spanning $(k-1)$-sphere but, in fact, contain no spanning tight component. Consequently, their conjecture about topological spheres points, more fundamentally, to $n/k$ as the threshold for global connectivity, a conjecture posed explictly by Illingworth, Lang, M\"uyesser, Parczyk, and Sgueglia~\cite{spanningspheres}. 

\begin{conjecture}[\cite{spanningspheres}]\label{conj:main}
    Every $n$-vertex $k$-uniform hypergraph with minimum codegree at least $n/k$ has a spanning tight component.
\end{conjecture}

\cref{conj:main} is known to hold for $k=2,3$: the $2$-uniform case is immediate from the fact that any two vertices in a Dirac graph have a common neighbour, while a short argument due to Mycroft (see~\cite[Corollary 7]{forcinglargetight}) yields the $3$-uniform case. 

Our main result resolves the next open case, namely $k=4$.

\begin{thm} \label{thm:main} Every $n$-vertex $4$-uniform hypergraph with minimum codegree at least $\lfloor n/4 \rfloor$ has a spanning tight component.
\end{thm}

This is exactly tight by the following slight modification of the construction from \cite{spanningsurfaces}. For the original construction, take an $n$-vertex $4$-graph $H$ whose vertex set is partitioned into four disjoint sets $V_1$, $V_2$, $V_3$, and $V_4$ with $\lfloor n/4\rfloor \leq |V_1| \leq \dots \leq |V_4| \leq \lceil n/4\rceil$. For $v \in V_i$, define $\ell(v) \coloneqq i$. Any four vertices $a, b, c, d$ form an edge in $H$ if and only if $\ell(a) + \ell(b) + \ell(c) + \ell(d) \equiv 1 \pmod{4}$. We now construct a modified $4$-graph $H'$ from $H$ by adding all edges $abcd$ with $\ell(a) \equiv \ell(b) \equiv \ell(c)-1 \equiv \ell(d) -1 \pmod{4}$. The minimum codegree of $H'$ is $\lfloor n/4 \rfloor -1$, witnessed by the triples in $V_1 \times V_3 \times V_4$. It is not hard to see that every tight component of $H'$ leaves all the vertices of some $V_i$ uncovered. 

Tight connectivity has also been intensively studied in connection with various hypergraph analogues of connected ($2$)-graph classes. Accordingly, a large body of work focuses on determining minimum codegree thresholds for the appearance of specific tightly connected spanning objects, such as tight Hamilton cycles \cite{rodlrucinskiszemeredi, rodlrucinskiszemeredi3unif}, powers of tight cycles \cite{posaseymour, squarehamiltonian} and their blow-ups \cite{hypergraphbandwidth}, and bounded degree tight trees~\cite{pssms}. All of these appear at substantially higher thresholds than $n/4$, and hence have no direct bearing on our result. Other degree thresholds for spanning tight components have also been investigated, including vertex degree \cite{spanningcomponentsandsurfaces} and supported codegree \cite{spanningspheres}. 

For the proof of \cref{thm:main}, we will work with a reformulation of the problem with a Ramsey-theoretic flavour. An \define{edge-colouring} of $K_n^{(3)}$ is a mapping $\chi: E(K_n^{(3)}) \to \mathbb{N}$, and the edges $e$ with $\chi(e) = i$ are said to be \define{$i$-coloured}. A colour $i \in \mathbb{N}$ is said to be \define{spanning} if the subgraph of $i$-coloured edges is spanning. A subgraph of $K_n^{(3)}$ is \define{monochromatic} if all its edges use the same colour. We will establish the following statement, which is easily shown to be equivalent to \cref{thm:main} (see the beginning of \cref{sec:main} below).

\begin{thm}\label{thm:main_colouring}
    Every edge-colouring $\chi : E(K_n^{(3)}) \to \mathbb{N}$ in which each edge lies in at least $\lfloor n/4\rfloor $ monochromatic copies of $K_4^{(3)}$ has a spanning colour.
\end{thm}

The proof of this theorem centres on certain special structures, which we call \define{$r$-configurations} (see \cref{defn:rconfig} below). These are edge-colourings of $K_r^{(3)}$ with no spanning colour, but with the property that the colouring induced on every proper subset of $V(K_r^{(3)})$ has a spanning colour. Assuming for a contradiction that $\chi$ has no spanning colour, it follows readily that $\chi$ must contains a non-empty subset inducing such a structure. 

We use this $r$-configuration in the following way. We first show, in the most challenging part of the proof, that in proving \cref{thm:main_colouring} it suffices to consider the case where $\chi$ uses no more than six distinct colours. A key feature of an $r$-configuration is that it necessarily uses at least as many colours as it has vertices; hence the structure found above has at most six vertices. Crucially, an $r$-configuration of such small size can be described rather precisely. We ultimately show that its structural properties are incompatible with the requirement that each of its edges lies in many monochromatic copies of $K_4^{(3)}$ in $\chi$, contradicting the assumption of \cref{thm:main_colouring}. 

\paragraph{Organisation.} We first discuss notation in \cref{sec:notation}. \cref{sec:main} is our main section, where we prove the equivalence of \cref{thm:main} and \cref{thm:main_colouring}, state the key lemmas that are needed for \cref{thm:main_colouring}, and finally give its proof (after a more thorough overview). \cref{sec:colorbounds} and \cref{sec:rconfigs} are devoted to the proofs of the key lemmas: the first essentially shows how to bound the number of colours appearing in $\chi$, whereas the second provides the aforementioned characterization of $r$-configurations on at most six vertices. We finish in \cref{sec:conclud} with some concluding remarks and open problems.

%%%%%%%%%%%%%%%%%%%%%%%%%%%%%%%%%%%%%

\section{Notation}\label{sec:notation}

Throughout the paper, we make use of the following notation. Given $k \geq 2$, a \define{$k$-graph} $H$ consists of a vertex set $V(H)$ together with an edge set $E(H) \subseteq \binom{V(H)}{k}$. We denote by \define{$e(H)$} the size of $E(H)$. Given a vertex set $X \subseteq V(H)$, we denote by \define{$H[X]$} the induced subgraph of $H$ on $X$. A subgraph $H'$ of $H$ is said to be \define{spanning} if every vertex in $V(H)$ is incident with an edge of $H'$. The $n$-vertex \define{complete $k$-graph}, denoted \define{$K_n^{(k)}$}, has $n$ vertices and satisfies $E(K_n^{(k)}) = \binom{V(K_n^{(k)})}{k}$. A $k$-uniform \define{tight path} is a sequent of distinct vertices $v_1, \dots, v_t$ with any $k$ consecutives one forming an edge; a \define{tight cycle} has all consecutive $k$-tuples as edges in a cyclic manner.

An edge-colouring of $K_n^{(k)}$ is a mapping $\chi: E(K_n^{(k)}) \to \mathbb{N}$. Given an edge-colouring $\chi$, a subgraph $H$ of $K_n^{(k)}$ is said to be \define{monochromatic} if there exists $i \in \mathbb{N}$ such that $\chi(e) = i$ for all $e \in E(H)$; it is said to be \define{rainbow} if $\chi(e) \neq \chi(e')$ for any distinct edges $e, e' \in E(H)$. Given $S \subseteq V(K_n^{(k)})$, the \define{induced subcolouring} of $\chi$ on $S$, denoted \define{$\chi[S]$}, is the mapping $\chi[S]: \binom{S}{k} \to \mathbb{N}$ satisfying $\chi[S](e) = \chi(e)$ for every edge $e \in \binom{S}{k}$.

For each $i \in \mathbb{N}$, we denote by~\define{$H_i^{\chi}$} the (monochromatic) subgraph of $K_n^{(k)}$ consisting of all $i$-coloured edges. A vertex $v \in V(K_n^{(k)})$ is \define{incident} with a colour $i \in \mathbb{N}$ if there exists an edge $e \in E(K_n^{(k)})$ such that~$v \in e$ and $\chi(e) = i$. The colour $i \in \mathbb{N}$ is said to be \define{spanning} if $H_i^{\chi}$ is a spanning subgraph of $K_n^{(k)}$. Finally, given a vertex $v \in V(K_n^{(k)})$, the \define{coloured link graph} of $v$, denoted \define{$L(v)$}, is the edge-colouring of the complete $(k-1)$-graph on vertex set $V(K_n^{(k)}) \setminus \{v\}$ in which each edge $e$ takes the same colour of $e \cup \{v\}$ under $\chi$. 

To simplify notation, we will sometimes write $v_1 \dots v_k$ as shorthand for the set $\{v_1, \dots, v_k\}$.

\section{Overview and proof of the main results}\label{sec:main}

In this section, we prove our main results, deferring the proof of certain key lemmas to later sections. We begin by establishing the equivalence of \cref{thm:main} and \cref{thm:main_colouring}. We then turn to \cref{thm:main_colouring}, outlining its proof while introducing various tools required for it. We end the section by completing its proof using these tools. 

\begin{proof}[Proof of the equivalence of \cref{thm:main} and \cref{thm:main_colouring}.]
    We first show that \cref{thm:main_colouring} implies \cref{thm:main}. Let $H$ be an $n$-vertex $4$-uniform hypergraph with minimum codegree at least $\lfloor n/4\rfloor$. Take a copy of $K_n^{(3)}$ on the same vertex set as $H$. It is easy to see that every triple in $E(K_n^{(3)})$ is included in the edges of a unique tight component of $H$. We may thus construct an edge-colouring $\chi: E(K_n^{(3)}) \to \mathbb{N}$ by associating a distinct colour to each tight component of $H$, and colouring each triple with the colour of its unique tight component. Each edge $e \in E(H)$ induces a monochromatic $K_4^{(3)}$ under $\chi$; hence, every triple in $E(K_n^{(3)})$ is contained in at least $\lfloor n/4\rfloor$ such monochromatic cliques by the codegree condition on $H$. Then, provided \cref{thm:main_colouring} holds, $\chi$ has a spanning colour. By construction, the edges of this colour class all come from the same tight component of $H$, which must therefore be spanning, proving \cref{thm:main}.  

    Next, we prove the other direction, namely that \cref{thm:main} implies \cref{thm:main_colouring}. Let $\chi$ be an edge-colouring of $K_n^{(3)}$ in which each edge is contained in at least $\lfloor n/4\rfloor$ monochromatic copies of $K_4^{(3)}$. We construct a $4$-graph $H$ on the same vertex set by adding an edge for each quadruple of vertices inducing a monochromatic $K_4^{(3)}$ under $\chi$. Then, the condition that each triple is contained in $\lfloor n/4\rfloor$ monochromatic copies of $K_4^{(3)}$ translates to $H$ having minimum codegree at least $\lfloor n/4\rfloor$. Provided \cref{thm:main} holds, $H$ must have a spanning tight component $C$. 

    We claim that all the triples included in edges of $C$ use the same colour. Consider any pair of triples $t, t' \in E(K_n^{(3)})$ and edges $e, e' \in E(C)$ with $t \subseteq e$ and $t' \subseteq e'$. By the tight connectivity of $C$, there exists a sequence $f_1, \dots, f_\ell \in E(C)$ with $f_1 = e$, $f_\ell = e'$, and $|f_i \cap f_{i+1}| =3$. Translating back to the edge-colouring $\chi$, this yields a sequence of monochromatic $K_4^{(3)}$s with any two consecutive ones intersecting in three vertices. The intersection property ensures that every clique in this sequence uses the same colour as the one before it and the one after. Hence, they all use the same colour, which forces $\chi(t) = \chi(t')$. This shows that all triples included in edges of $C$ use the same colour, which is spanning in $\chi$ because $C$ is spanning in $H$, thereby proving \cref{thm:main_colouring}.
\end{proof}

We now turn to \cref{thm:main_colouring}, whose proof is outlined below together with the lemmas required for it. In the argument, we assume the existence of an edge-colouring $\chi$ of $K_n^{(3)}$ with no spanning colour, but such that every edge extends to at least $\lfloor n/4\rfloor$ monochromatic copies of $K_4^{(3)}$, and we derive a contradiction. 

The first step in the proof is to bound the number of colours appearing in $\chi$. Using the condition that every edge lies in many monochromatic copies of $K_4^{(3)}$, it is not hard to see that this number is bounded above by a constant. Our argument, however, demands a particularly strong upper bound: namely, that no more than $6$ colours appear. To this end, in \cref{sec:colorbounds} we prove the following lemma about edge-colourings of complete ($2$-)graphs, which may be of independent interest. 

\begin{lem}\label{lem:UB_colours} Any edge-colouring $\chi: E(K_n) \to \mathbb{N}$ in which every edge lies in at least $\abund$ monochromatic triangles uses at most 5 colours.    
\end{lem}

The way to relate \cref{lem:UB_colours} to our edge-colouring $\chi$ of $K_n^{(3)}$ is via the notion of the \emph{coloured link graph} $L(v)$ of each vertex $v \in V(K_n^{(3)})$. Recall from \cref{sec:notation} that $L(v)$ is defined as the edge-colouring of the complete ($2$-uniform) graph on vertex set $V(K_n) \setminus \{v\}$ in which each edge $xy$ inherits the colour of $vxy$ under $\chi$. For every $v \in V(K_n^{(3)})$, the properties of $\chi$ readily imply that $L(v)$ satisfies the assumption of \cref{lem:UB_colours}. This shows that $L(v)$ uses at most $5$ colours; in other words, $v$ is incident with at most $5$ colours in $\chi$. 

To convert this into a global bound on the number of colours, we perform an iterative procedure on $\chi$ which merges colours together whenever they jointly avoid a vertex. This modifies the colouring while preserving all its relevant properties. At its conclusion, every vertex is incident with all but at most one colour appearing in $\chi$, and hence the total number of colours is at most $6$.

The bound on the number of colours is used in the following way. First, we exploit the absence of a spanning colour in $\chi$ to show it contains the following structure as an induced subcolouring.

\begin{defn}[$r$-configuration]\label{defn:rconfig}
Let $r \geq 3$. An edge-colouring of~$K_r^{(3)}$ is said to be an \define{$r$-configuration} if it satisfies the following properties.
\begin{enumerate}[label = \itmarab{C}]
    \item\label{config:1} the colouring $\chi$ has no spanning colour; and
    \item\label{config:2} for every proper subset $S \subsetneq V(K_r^{(3)})$ with $|S| \geq 3$, the induced subcolouring $\chi[S]$ has a spanning colour.
\end{enumerate}
\end{defn}

Every $r$-configuration necessarily has edges of at least $r$ distinct colours; indeed, there are colours spanning all its $(r-1)$-element subsets by \ref{config:2}, and they must all be distinct. Hence the bound on the total number of colours implies that the $r$-configuration we found in $\chi$ has $r \leq 6$ vertices. At the other end of the range, it must also satisfy $r \geq 4$, which readily follows from \ref{config:1}.

It turns out that $r$-configurations of such small size can be characterized rather precisely. In fact, up to relabelling the vertices and colours, there is a unique $r$-configuration for each $r \in \{4,5\}$. Under an additional colouring constraint arising from our assumption in \cref{thm:main_colouring}, there is also a unique $6$-configuration (specifically, the constraint in question is that every pair of vertices lies in edges of at most three colours, see \cref{obs:basic_facts}\ref{basic_obs:fact2}). Fully describing the structures arising from the cases $r= 5,6$ is not necessary for the proof of \cref{thm:main_colouring}; for our purposes, it suffices to establish \cref{lem:config_classif} below, which captures their key properties. Its proof is given in \cref{sec:rconfigs}. The unique $5$-configuration and $6$-configuration are depicted in \cref{fig:configurations}.

\begin{lem}\label{lem:config_classif}
The following hold for every $r$-configuration $\chi$.
\begin{enumerate}[label = \textnormal{(\roman*)}]
    \item\label{lem:config_classif:1} If $r=4$, then $\chi$ is rainbow.
    \item\label{lem:config_classif:2} If $r=5$, then $\chi$ has at most two edges of each colour.
    \item\label{lem:config_classif:3} If $r=6$ and $\chi$ uses exactly $6$ colours, then $\chi$ has at most $5$ edges of each colour. Additionally, if every pair of vertices lies in edges of at most $3$ distinct colours, then $\chi$ contains no monochromatic $K_4^{(3)}$.
\end{enumerate}
\end{lem}

This lemma may then be applied to the $r$-configuration found in the proof of \cref{thm:main_colouring} to show that it satisfies either \ref{lem:config_classif:1}, \ref{lem:config_classif:2}, or \ref{lem:config_classif:3}. All the edges inside this structure lie in at least $\lfloor n/4\rfloor$ monochromatic copies of $K_4^{(3)}$; hence, by averaging, some vertex forms monochromatic cliques with approximately one fourth of these edges. We finish the proof of \cref{thm:main_colouring} by showing that this fact, in combination with whichever option among \ref{lem:config_classif:1}, \ref{lem:config_classif:2}, and \ref{lem:config_classif:3} holds, yields a contradiction. 

The last ingredient we require for executing the above strategy is the following observation, used frequently in the proof of \cref{thm:main_colouring}. 

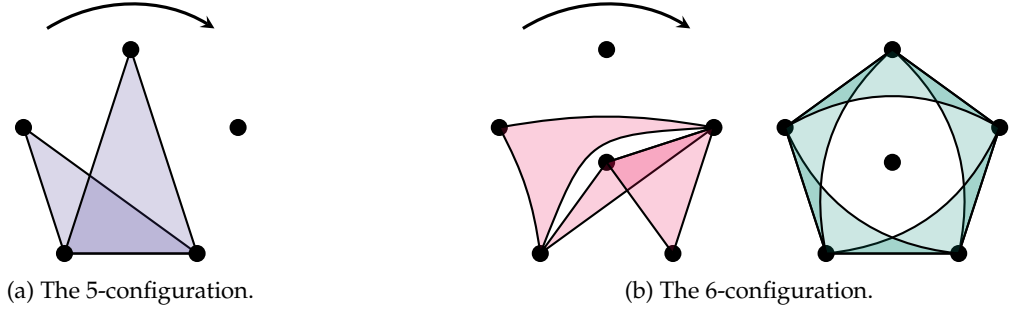
\begin{figure}[t]
    \begin{subfigure}{0.49\textwidth}
    \centering
    \begin{tikzpicture}[thick, scale = 1.5]
    %labels
    
   \coordinate (v1) at (90:1);
    \coordinate (v5) at (162:1);
    \coordinate (v4) at (234:1);
    \coordinate (v3) at  (306:1);
    \coordinate (v2) at (18:1);
    \coordinate (o) at (0, 0);

    \filldraw[fill=BlueViolet, fill opacity = 0.2, draw=black, thick] 
        (v3) -- (v4) -- (v5) -- cycle;
    \filldraw[fill=BlueViolet, fill opacity = 0.2, draw=black, thick] 
        (v3) -- (v4) -- (v1) -- cycle;

    \tikzstyle{every node}=[circle, draw, fill, inner sep=2pt, minimum width=2pt]
    \draw (v1) node {};
    \draw (v2) node {};
    \draw (v3) node {};
    \draw (v4) node {};
    \draw (v5) node {};

    \draw[very thick, -stealth] (122:1.4) arc (122:58:1.4);

\end{tikzpicture}
\caption{The $5$-configuration.}
\end{subfigure}
\begin{subfigure}{0.49\textwidth}
\centering
       \begin{tikzpicture}[thick, scale=1.5]
    %labels
    
    \coordinate (v1) at (90:1);
    \coordinate (v5) at (162:1);
    \coordinate (v4) at (234:1);
    \coordinate (v3) at  (306:1);
    \coordinate (v2) at (18:1);
    \coordinate (v6) at (0, 0);

    \tikzstyle{every node}=[circle, draw, fill, inner sep=2pt, minimum width=2pt]
    \draw (v1) node {};
    \draw (v2) node {};
    \draw (v3) node {};
    \draw (v4) node {};
    \draw (v5) node {};
    \draw (v6) node {};

    \filldraw[fill=OrangeRed, fill opacity = 0.2, draw=black, thick] 
        (v5) to[bend left = 10] (v2) .. controls (126:0.3) .. (v4)     
        to [bend right = 10] (v5); 

    \filldraw[fill=OrangeRed, fill opacity = 0.2, draw=black, thick] 
        (v6) -- (v2) -- (v4) -- cycle;

    \filldraw[fill=OrangeRed, fill opacity = 0.2, draw=black, thick] 
        (v6) -- (v2) -- (v3) -- cycle;

      \draw[very thick, -stealth] (122:1.4) arc (122:58:1.4);

\end{tikzpicture} \hspace{15pt} \begin{tikzpicture}[thick, scale=1.5]
    %labels
    
    \coordinate (v1) at (90:1);
    \coordinate (v5) at (162:1);
    \coordinate (v4) at (234:1);
    \coordinate (v3) at  (306:1);
    \coordinate (v2) at (18:1);
    \coordinate (v6) at (0, 0);

    \tikzstyle{every node}=[circle, draw, fill, inner sep=2pt, minimum width=2pt]
    \draw (v1) node {};
    \draw (v2) node {};
    \draw (v3) node {};
    \draw (v4) node {};
    \draw (v5) node {};
    \draw (v6) node {};

    \filldraw[fill=PineGreen, fill opacity = 0.2, draw=black, thick] 
        (v3) to[bend right = 30] (v1) -- (v2) -- cycle; 
    \filldraw[fill=PineGreen, fill opacity = 0.2, draw=black, thick] 
        (v4) to[bend right = 30] (v2) -- (v3) -- cycle; 
    \filldraw[fill=PineGreen, fill opacity = 0.2, draw=black, thick] 
        (v5) to[bend right = 30] (v3) -- (v4) -- cycle; 
    \filldraw[fill=PineGreen, fill opacity = 0.2, draw=black, thick] 
        (v1) to[bend right = 30] (v4) -- (v5) -- cycle; 
    \filldraw[fill=PineGreen, fill opacity = 0.2, draw=black, thick] 
        (v2) to[bend right = 30] (v5) -- (v1) -- cycle; 

    \draw[white, very thick, -stealth] (122:1.4) arc (122:58:1.4);

    \end{tikzpicture}

    \caption{The $6$-configuration.}\label{fig:6config}
        
    \end{subfigure}

    \caption{ A schematic depiction of the unique $5$-configuration and the unique $6$-configuration with no pair of vertices lying in edges of $4$ distinct colours. The $5$-configuration has as its colour classes the five rotations of the tight path in (a). The $6$-configuration has as one of its colour classes the tight $5$-cycle on the right in (b), while the remaining five classes correspond to the five rotations of the tight path on the left.}\label{fig:configurations}

    \end{figure}

\needspace{6\baselineskip}
\begin{obs}\label{obs:basic_facts} Every edge-colouring $\chi: K_n^{(3)} \rightarrow \mathbb{N}$ in which every edge lies in at least $\lfloor n/4 \rfloor$ monochromatic copies of $K_4^{(3)}$ satisfies the following. 
\begin{enumerate}[label = \textnormal{(\roman*)}]
    \item\label{basic_obs:fact1} $\chi$ contains no rainbow copy of $K_4^{(3)}$.
    \item\label{basic_obs:fact2} Every pair of vertices lies in edges of at most 3 distinct colours.
\end{enumerate}  
\end{obs}

\begin{proof}
Throughout the proof, for each edge $e \in E(K_n^{(3)})$, we use $K(e)$ to denote the set of vertices $v \in V(K_n^{(3)})$ such that $e \cup \{v\}$ induces a monochromatic $K_4^{(3)}$ in $\chi$ (so that $|K(e)| \geq \lfloor n/4 \rfloor$ in general).

For \ref{basic_obs:fact1}, suppose for a contradiction that $\chi$ contains a rainbow $K_4^{(3)}$ on vertex set $U \subseteq V(K_n^{(3)})$. Let $e_1,...,e_4$ be the triples in $\binom{U}{3}$, each having been assigned a distinct colour under $\chi$. Since $\chi[U]$ is rainbow, we have $K(e_i) \cap U = \emptyset$ for each $i \in [4]$. Then the minimum containment condition yields $\sum_{i \in [4]} |K(e_i)| \geq 4\lfloor n/4\rfloor \geq  n -3 > |V(K_n^{(3)}) \setminus U|$. By pigeonhole, there exists a vertex $x \in V(K_n^{(3)}) \setminus U$ extending distinct edges $e_i$ and $e_j$ to monochromatic $K_4^{(3)}$s of distinct colours. Observe that $|e_i \cap e_j| = 2$. This yields a contradiction as the triple $(e_i \cap e_j) \cup \{x\}$ lies in both of these cliques but only takes one colour. 
    
    For \ref{basic_obs:fact2}, suppose for a contradiction that a pair of vertices $uv$ lies in four edges $e_1$, $e_2$, $e_3$, and $e_4$, each using a distinct colour. Note that $K(e_i) \cap e_j = \emptyset$ for each $i, j \in [k]$. Similarly to the above, we have $\sum_{i \in [k]} K(e_i) \geq n-3 > |V(K_n^{(3)}) \setminus \bigcup_{i \in [4]} e_i|$. By pigeonhole, there exists a vertex $x \in V(K_n^{(3)}) \setminus \bigcup_{i \in [4]} e_i$ extending distinct edges $e_j$ and $e_\ell$ to monochromatic $K_4^{(3)}$s of distinct colours. This leads to the same contradiction as for \ref{basic_obs:fact1}, since the triple $uvx$ is shared by both these cliques.
\end{proof}

We finish this section by proving our main result about edge-colourings of $K_n^{(3)}$. It is worth remarking that our strategy bears similarities to that used by Mycroft for the $3$-uniform case of \cref{conj:main} (see \cite[Proposition 6 and Corollary 7]{forcinglargetight}). In analogy with our approach, the proof uses the tight component structure to define an edge-colouring of $K_n$. One then shows that, if no colour is spanning, the colouring either contains a rainbow triangle, which plays the same role as our $r$-configurations, or uses only two colours in total. In both cases, a contradiction is obtained.

\begin{proof}[Proof of \cref{thm:main_colouring}]Suppose for a contradiction that there is an edge-colouring $\chi$ of $K_n^{(3)}$ in which every edge lies in at least $\lfloor n/4\rfloor$ monochromatic $K_4^{(3)}$s and with no spanning colour. Let $V$ be the vertex set of $K_n^{(3)}$. Recall that $H_i = H_i^{\chi}$ is the $3$-graph of $i$-coloured edges, and that, for each $v \in V$, the coloured link graph of $v$ with respect to $\chi$ is denoted $L(v)$.

We modify the colouring $\chi$ by the following iterative procedure. While there exists a vertex $v \in V$ and distinct colours $i, i' \in \mathbb{N}$ such that $e(H_i), e(H_{i'}) > 0$ and $v$ is not incident with any $i$-coloured or $i'$-coloured edges, re-colour every edge of colour $i'$ with colour $i$. We continue this process until no such triple $(v, i, i')$ remains, and re-use $\chi$ to denote the resulting colouring. Note that both the absence of a spanning colour and the bound that every edge lies in $\lfloor n/4\rfloor$ monochromatic $K_4^{(3)}$s are preserved throughout the procedure. Moreover, the final colouring has the additional property that every edge is incident with all but at most one colour appearing in $\chi$. Relabelling the colours if necessary, we may therefore assume that

\begin{enumerate}[label = (*)]
    \item\label{prop:k-1coloursateachvtx} for some $k \geq 1$, the colouring satisfies $\chi :E(K_n^{(3)}) \to [k]$ and every vertex $v \in V$ is incident with at least $k-1$ distinct colours in $[k]$.
\end{enumerate}

Now we bound $k$ from above. For each vertex $v \in V$, observe that $L(v)$ is an edge-colouring of $K_{n-1}$ in which each edge lies in at least $\lfloor n/4\rfloor$ monochromatic triangles.
By \cref{lem:UB_colours}, $L(v)$ uses at most 5 distinct colours. Together with \ref{prop:k-1coloursateachvtx}, this implies $k \leq 6$. 

Since $\chi$ contains no spanning colour, for each $i \in [k]$ we may choose a vertex $v_i$ not incident with colour $i$. Observe that \ref{prop:k-1coloursateachvtx} ensures that $v_i \neq v_{i'}$ for distinct $i, i' \in [k]$. Let $S = \{v_1, \dots, v_k\}$. We have $|S| \geq 4$; if not, then $S \subseteq e$ for some $e \in K_n^{(3)}$, and thus every vertex in $S$ is incident with the colour $\chi(e) \in [k]$, contradicting the fact that $v_{\chi(e)} \in S$. Finally, note that $\chi[S]$ has no spanning colour. Indeed, if a colour $i \in [k]$ were to span $\chi[S]$, then $v_i \in S$ would be incident with colour $i$, a contradiction.

Let $T \subseteq S$ be a minimal subset (under inclusion) with the property that $|T| \geq 4$ and $\chi[T]$ has no spanning colour. Let $r\coloneqq |T|$. It is easy to see that $\chi[T]$ is an $r$-configuration (indeed, property \ref{config:1} is true by construction, whereas the minimality of $T$ ensures property \ref{config:2}). Moreover, we have $4 \leq r \leq k \leq 6$. Observe that, by \cref{obs:basic_facts}\ref{basic_obs:fact2}, every pair of distinct vertices in $V$ lies in edges of at most 3 distinct colours, and this remains true if we restrict to $\chi[T]$.  Thus, by \cref{lem:config_classif} one of the following must occur.
\begin{enumerate}[label = (\alph*)]
    \item\label{case4} $r=4$ and $\chi[T]$ is a rainbow $K_4^{(3)}$;
    \item\label{case5} $r=5$ and every colour appears at most twice in $\chi[T]$; or
    \item\label{case6} $r=6$, $\chi[T]$ contains no monochromatic $K_4^{(3)}$ and has at most $5$ edges of any given colour.
\end{enumerate}

Observe that case \ref{case4} immediately yields a contradiction by \cref{obs:basic_facts}\ref{basic_obs:fact1}. To finish the proof, we will show that both cases \ref{case5} and \ref{case6} also lead to a contradiction. These are treated separately below. First, we introduce some more useful notation and an observation. 

After relabelling the colours if necessary, and changing the subscripts of the vertices in $T$ accordingly, we may assume that $T = \{v_1, \dots, v_r\}$ and that each $v_i$ is not incident with colour $i \in [k]$ (whenever we relabel colours below, we always assume that the subscripts of the $v_i$ are adjusted to maintain this property). For each choice of distinct $h, i, j \in [r]$, let us write $K(v_hv_i v_j)$ to denote the set of vertices $x \in V$ such that $x v_hv_iv_j$ induces a monochromatic $K_4^{(3)}$. Note that in both cases \ref{case5} and \ref{case6} we have that $\chi[T]$ contains no monochromatic $K_4^{(3)}$, and so in general $K(v_h v_i v_j) \subseteq V \setminus T$. Also, let us record the following observation which will be used several times below.

\begin{enumerate}[label = ($\dagger$)]
    \item\label{prop:disjoint} if $K(e) \cap K(e') \neq \emptyset$ for distinct $e, e' \in \binom{T}{3}$ with $|e \cap e'| = 2$, then $\chi(e) = \chi(e')$. 
\end{enumerate}
Let us see why this holds. Observe that if $x \in K(e) \cap K(e')$ and $|e \cap e'| = 2$, then the edge $(e \cap e') \cup \{x\}$ lies in two distinct monochromatic $K_4^{(3)}$s, namely those induced by $e \cup \{x\}$ and $e \cup \{x'\}$. This implies that $\chi(e) = \chi((e \cap e')\cup \{x\}) = \chi(e')$, as desired.

\paragraph{Case \ref{case5}.} Here, we have $T = \{v_1, \dots, v_5\}$ and $\chi[T]$ contains at most two edges of any given colour. Recall that all the colours appearing in $\chi$ belong to $[6]$, and so all but at most two edges in $\binom{T}{3}$ use colours in $[5]$. Let $\mathcal{E} \subseteq \binom{T}{3}$ be the subset of edges taking colours in $[5]$, so that $|\mathcal{E}| \geq \binom{5}{2} - 2 \geq 8$. Then 
\[\sum_{e \in \mathcal{E}} K(e) \geq 8 \lfloor n/4 \rfloor > 2(n-5) = 2 |V\setminus T|. \]
Thus, by averaging over all vertices in $V \setminus T$, we infer the existence of some $x \in V \setminus T$ and distinct edges $e_1, e_2, e_3 \in \mathcal{E}$ with $x \in K(e_i)$ for each $i \in [3]$. 

Using the fact that $|T| = 5$ and $|e_i| = 3$, it is easy to see that for two of these three edges, say~$e_1$ and~$e_2$, we have $|e_1 \cap e_2| = 2$.
By \ref{prop:disjoint}, this implies $\chi(e_1) = \chi(e_2)$. In turn, this forces $|e_1 \cap e_3|, |e_2 \cap e_3| \leq 1$, as otherwise $e_1, e_2,$ and $e_3$ must all use the same colour again by \ref{prop:disjoint}, which contradicts \ref{case5}. The only way this can occur is if $|e_1 \cap e_3| = |e_2 \cap e_3| = 1$. 
By relabelling the colours if necessary, we may assume that $e_1 = v_1v_2v_3, e_2 = v_1 v_2 v_4$ and $e_3 = v_3v_4v_5$. However, every $e_i \in \mathcal{E}$ takes a colour in $[5]$, and so, recalling that each $v_i$ is not incident with colour $i$, we must have $\chi(e_1) = \chi(e_2) = 5$. For the same reason, we have $\chi(e_3) \in [2]$; by switching the labels of the colours $1$ and $2$ if necessary, we may assume that $\chi(e_3) = 1$.

Since $x \in K(e_1) \cap K(e_2)$, it follows that the edges $xv_1v_2, xv_1v_3, xv_2v_3, xv_1v_4, xv_2v_4$ all use colour~$5$. Similarly, $x \in K(e_3)$ implies that $xv_3v_4, xv_4v_5, xv_3v_5$ use colour $1$. We claim that $\chi(xv_1v_5) = \chi(v_1v_3v_5) = \chi(v_1v_4v_5)$. Indeed, if $\chi(xv_1v_5) \neq \chi(v_1v_3v_5)$, then $xv_1v_3v_5$ forms a rainbow $K_4^{(3)}$ as $\chi(xv_1 v_3) = 5$, $\chi(x v_3v_5) = 1$ and both $xv_1v_5$ and $v_1v_3v_5$ cannot take a colour in $\{1,5\}$ (recall that each $v_i$ is not incident with any $i$-coloured edges). For the same reason, if $\chi(x v_1 v_5) \neq \chi(v_1v_4v_5)$, then $x v_1 v_4 v_5$ forms a rainbow $K_4^{(3)}$. As either of these options yields a contradiction together with \cref{obs:basic_facts}\ref{basic_obs:fact1}, this shows that $\chi(v_1 v_3v_5) = \chi(v_1 v_4v_5) = c$ for some $c \in \{2,6\}$. 

Recall that each colour appears at most twice in $\chi [T]$, and so no other edge in $\chi[T]$ uses colour $c$. For the same reason, no edge except for $v_1v_2v_3$ and $v_1 v_2 v_4$ uses colour $5$. Since the quadruple $v_1v_2v_3v_5$ does not induce a rainbow $K_4^{(3)}$, we must then have $\chi(v_1v_2 v_5) = \chi(v_2v_3v_5)$. Similarly, the fact that $v_1 v_2 v_4 v_5$ does not induce a rainbow $K_4^{(3)}$ implies that $\chi(v_1 v_2 v_5) = \chi(v_2v_4v_5)$. But then the edges $v_1 v_2 v_5, v_2 v_3 v_5$ and $v_2 v_4 v_5$ all take the same colour, contradicting \ref{case5}.

\paragraph{Case \ref{case6}.} Here, we have $T = \{v_1, \dots, v_6\}$. We argued previously that $r \leq k \leq 6$, and so in this case we have $k = 6$. Then \ref{prop:k-1coloursateachvtx} implies that every $v \in V$ is incident with precisely $5$ colours in $\chi$. We can now define a partition $V_1 \cup \dots \cup V_6$ of $V$, where each vertex in $V_i$ is incident with all colours except for $i$. By construction, $v_i$ is not incident with $i$, and thus $v_i \in V_i$. 

Observe that
\[\sum_{e \in \binom{T}{3}} K(e) \geq \binom{6}{3} \lfloor n/4 \rfloor > 5(n- 6) = 5 |V\setminus T|.\]
Thus, by averaging over all vertices in $V \setminus T$, we infer the existence of some $x \in V \setminus T$ such that $x \in K(e_i)$ for $6$ distinct edges $e_1, \dots, e_6 \in \binom{T}{3}$. 

Let $j \in [6]$ satisfy $x \in V_j$. We claim that $v_j \in e_i$ for each $j \in [6]$. Indeed, suppose otherwise and let $e \in \binom{T}{3}$ satisfy $x \in K(e)$ and $v_j \notin e$. Then, letting $T' = (T \setminus \{v_j\}) \cup \{x\}$, we have that $\chi[T']$  contains a monochromatic $K_4^{(3)}$ (namely, $e \cup \{x\}$). However, for each colour $i \in [6]$ there is a vertex in $T'$ not incident with colour $i$, and thus $\chi[T']$ has no spanning colour. Let $T'' \subseteq T'$ be a minimal subset under inclusion such that $|T''| \geq 4$ and $\chi[T'']$ has no spanning colour. By the same argument used above, $\chi[T'']$ is an $r'$-configuration for some $4 \leq r' \leq 6$. If $r' = 4$, then we obtain a rainbow $K_4^{(3)}$ by \cref{lem:config_classif}, which yields a contradiction by \cref{obs:basic_facts}\ref{basic_obs:fact1}. If $r' = 5$, then this reproduces the setting of case \ref{case5} and so repeating the argument used there yields another contradiction. Then $r' = 6$ and so $T'' = T'$. Thus, $\chi[T']$ is a $6$-configuration containing a monochromatic $K_4^{(3)}$. By \cref{lem:config_classif}\ref{lem:config_classif:3}, $\chi[T']$ must then contain a pair of vertices lying in edges of $4$ distinct colours, which contradicts \cref{obs:basic_facts}\ref{basic_obs:fact2}. This shows that $v_j \in e_i$ for each $i \in [6]$.

Observe, by the second part of \ref{case6}, that the edges $\{e_i\}_{i \in [6]}$ cannot all use the same colour. Since $x \in K(e_i)$ for each $i \in [6]$, any pair of edges $e_i, e_{i'}$ with $\chi(e_i) \neq \chi(e_{i'})$ satisfies $|e_i \cap e_{i'}| \leq 1$ by \ref{prop:disjoint}. In particular, for any such pair we have $e_i \cap e_{i'} = \{v_j\}$. If there were three edges $e_i, e_{i'}, e_{i''}$ of distinct colours, we would then have $|e_i \cup e_{i'} \cup e_{i''}| = 7$, which contradicts the fact that $e_i, e_{i'}, e_{i''} \subseteq T$. So we may assume that the edges in $\{e_i\}_{i \in [6]}$ use precisely two colours, say $c_1$ and $c_2$. For $j \in [2]$, let $E_j$ be the set of edges in $\{e_i\}_{i \in [6]}$ taking colour $c_j$. By the earlier observation that edges from $E_1$ can only intersect those in $E_2$ at $v_j$, it follows from the inclusion-exclusion principle that \begin{equation*}
    \big| \bigcup E_1 \cup \bigcup E_2\big| =  \big|\bigcup E_1\big| + \big|\bigcup E_2 \big| - 1.
\end{equation*}
Since $E_i$ is a set of triples, if $|E_i| \geq 2$ we must have $|\bigcup E_i| \geq 4$. Thus, if $|E_1|, |E_2| \geq 2$, then the previous equation implies that $|\bigcup E_1 \cup \bigcup E_2| \geq 7$, contradicting the fact that each $e_i \subseteq T$ and $|T| = 6$. Then, by switching the labels of $c_1$ and $c_2$ if necessary, we may assume that $|E_1| = 1$. But then $|E_2| = 6 - 1 = 5$ can be easily seen to imply $|\bigcup E_2| \geq 5$. The above equation thus yields $|\bigcup  E_1 \cup \bigcup E_2| \geq 3 + 5 - 1 \geq 7$, giving another contradiction and finishing the proof. \end{proof}

\section{Bounding the number of colours}\label{sec:colorbounds}

In this section, we prove \cref{lem:UB_colours}. For better readability, we define an edge-colouring of~$K_n$ to be \define{$r$-abundant} if every edge of $K_n$ lies in at least $r$ monochromatic triangles. Thus, \cref{lem:UB_colours} states that every $\abund$-abundant colouring of $K_n$ uses at most five colours. Recall that in the proof of \cref{thm:main_colouring}, this lemma is applied to the coloured link graphs of vertices in a suitably modified colouring~$\chi$ of $K_n^{(3)}$, to obtain a strong upper bound on the number of colours used by $\chi$.  

Before sketching an outline of the proof, we define some notations and make some simple observations regarding $\abund$-abundant colourings. Given an edge-colouring of $K_n$, we use~\define{$\mathcal{C}$} to denote the set of colours used by at least one edge of $K_n$. Further, for any vertex $v \in V(K_n)$, we say that a colour $i \in \mathcal{C}$ is \define{incident} with a vertex $v$, if $v$ is incident with an $i$-coloured edge. The set of colours in $\mathcal{C}$ incident with $v$ is denoted by \define{$\mathcal{C}_v$}. Finally, for each colour $i \in \mathcal{C}_v$ and any subset $T \subseteq V(K_n)$, we use \define{$N_i(v; T)$} to denote the set of vertices in $T$ that form an $i$-coloured edge with the vertex~$v$. We use \define{$\deg_i(v; T)$} to denote the size of the set $N_i(v;T)$, and drop the set~$T$ from the notation when $T = V(K_n)$. Then, we have the following observation. 

\begin{obs} \label{obs:abund_col}
     Let $\chi$ be a $\abund$-abundant colouring of $K_n$. Then, the following hold for every vertex $v \in V(K_n)$ and for each colour $i \in \mathcal{C}_v$ incident to $v$.
    \begin{enumerate}[label = \itmarab{A}]
        \item \label{itm:AC1} For any vertex $u \in N_i(v)$, we have $\deg_i\bigl(u; N_i(v)\bigr) \geq \abund$.
        \item \label{itm:AC2} We have $\deg_i(v) \geq  \abund + 1 > n/4$, and thus $|\mathcal{C}_v| \leq 3$.
    \end{enumerate}
\end{obs}
\begin{proof} Given any vertex $v \in V(K_n)$ and a colour $i \in \mathcal{C}_v$, let $u \in N_i(v)$ be any $i$-coloured neighbour of $v$. By the abundance property of $\chi$, there is a set $U \subseteq V(K_n)$ of at least $\abund$ vertices of $K_n$ that form $i$-coloured monochromatic triangles with the edge $uv$. Hence, we have that $U \subseteq N_i(u) \cap N_i(v)$, thereby proving \ref{itm:AC1}. For \ref{itm:AC2}, by definition of $\mathcal{C}_v$, we have $N_i(v) \neq \emptyset$ for any colour $i \in \mathcal{C}_v$. Let $u$ be a vertex in $N_i(v)$. Then, we have that $\deg_i(v) \geq  1 + \deg_i\bigl(u; N_i(v)\bigr) > \abund$ for each $i \in \mathcal{C}_v$. Finally, note that $\deg_i(v) > \abund$ implies $\deg_i(v) > n/4$ since $\deg_i(v)\in \mathbb{N}$. In turn, this gives $|\mathcal{C}_v| \leq 3$, thereby proving \ref{itm:AC2}. 
\end{proof}

The proof of \cref{lem:UB_colours}, deferred to the end of this section, relies on finding a (possibly unbalanced) blow-up of a rainbow $K_4$ in the given $\abund$-abundant colouring $\chi$ of $K_n$. Note that given an edge-colouring $\chi$ of $K_4$, a \define{blow-up} of this colouring is defined as a collection of four disjoint subsets $\{T_v\}_{v \in V(K_4)}$ of vertices such that, for each pair of vertices $u, v \in V(K_4)$, the edges between $T_u$ and $T_v$ form a $\chi(uv)$-coloured monochromatic complete bipartite graph.  

In the proof of \cref{lem:UB_colours}, we first show that if an $\abund$-abundant colouring~$\chi$ of $K_n$ uses at least six colours, then it contains a rainbow~$K_4$. This allows us to consider a vertex-maximal blow-up of this rainbow~$K_4$ in the given colouring. We will be interested in bounding the size of this blow-up, and show that it can be neither too large nor too small. To obtain these bounds, we look at certain smaller rainbow structures in the colouring~$\chi$. We show that the $\abund$-abundance property of~$\chi$, applied to the edges of these substructures, imposes conflicting constraints on the size of the blow-up, thereby proving the lemma by contradiction.  

\subsection{The existence of vertices with \texorpdfstring{$|\mathcal{C}_v| = 3$}{Cv=3}}

In order to find a rainbow $K_4$ in the given colouring $\chi$, we need to find vertices incident with exactly three colours (recall that by \ref{itm:P2}, we have $|C_v| \leq 3$). Additionally for each of these vertices, we require the existence of a rainbow triangle coloured using colours in $\mathcal{C} \setminus \mathcal{C}_v$, and such that it is \emph{transversal} with respect to the collection of sets $\{N_i(v)\}_{i \in \mathcal{C}_v}$. Here, given a collection of pairwise disjoint sets $S_1, S_2, S_3 \subseteq V(K_n)$, a triangle $T$ in $K_n$ is said to be \define{transversal} with respect to $\{S_1, S_2, S_3 \}$ if we have $|V(T) \cap S_i| = 1$ for each $i \in [3]$. The existence of such vertices having desirable properties is shown in \cref{lem:CF37needsnospan}. 

In \cref{lem:pivot_basic_props}, we first collect some useful properties satisfied by vertices incident with three colours. For any vertex~$v$ with $|\mathcal{C}_v| = 3$, and for each colour $i \in \mathcal{C}_v$, we use \define{$\gamma_i(v)$} to denote the surplus quantities $\gamma_i(v) \coloneqq \deg_i(v) - \abund$. Note that by \cref{obs:abund_col}, we have $\gamma_i(v) \geq 1$ for each vertex~$v$ and $i \in \mathcal{C}_v$. Moreover, while $\gamma_i(v)$ depends on~$v$ and its colour set~$\mathcal{C}_v$, the sum~$\sum_{i \in \mathcal{C}_v}\gamma_i(v)$ is independent of the choice of~$v$. We record this invariance in the following equation, which is a consequence of the fact that $|\mathcal{C}_v| = 3$, and will be used frequently in the proofs below.
    \begin{equation}
    \label{eqn:gamma_sum}
        \sum_{i \in \mathcal{C}_v} \gamma_i(v) = n- 1 - 3\abund = \abund + \bigl((n+1)\bmod{4}\,\bigr) - 2.
    \end{equation}

\begin{lem}
\label{lem:pivot_basic_props}
    Let $\chi$ be a $\abund$-abundant colouring of $K_n$ and let $v \in V(K_n)$ be any vertex that satisfies $|\mathcal{C}_v| = 3$. Then, the following hold for any colour $j \in \mathcal{C} \setminus\mathcal{C}_v$.
    \begin{enumerate}[label = \itmarab{P}]
        \item \label{itm:P1} For each $k \in \mathcal{C}_v$, if a vertex $u \in N_k(v)$ is incident with the colour~$j$, then there exists a $j$-coloured edge~$uw$ with $w \notin N_k(v)$. 
        \item \label{itm:P2} If $\mathcal{C}_v = \{a,b,c\}$, then any $j$-coloured edge $uw \in N_a(v) \times N_b(v)$ is contained in at least $\gamma_c(v)\geq 1$ monochromatic triangles that are transversal with respect to $\{N_a(v), N_b(v), N_c(v)\}$.
    \end{enumerate}
    In particular, for each $k \in \mathcal{C}_v$, there are at least $\gamma_k(v)$ vertices in $N_k(v)$ that are incident with the colour $j$ .
\end{lem}

\begin{proof}
    Let $\chi$ be a $\abund$-abundant colouring of $K_n$ and let $v$ be a given vertex with $|\mathcal{C}_v| = 3$. Set $\gamma_i \coloneqq \gamma_i(v)$. Note that the sets $\{N_i(v)\}_{i \in \mathcal{C}_v}$ tripartition $V(K_n) \setminus \{v\}$. Moreover, by \ref{itm:AC2}, the following holds for each $i \in \mathcal{C}_v$.
    \[n/4 < |N_i(v)| \leq n-1 - 2 \cdot n/4 < n/2. \] 
    Now we show \ref{itm:P1}. Let the colours $j \in \mathcal{C} \setminus \mathcal{C}_v$ and $k \in \mathcal{C}_v$ be given. Let $u \in N_k(v)$ be any vertex in $K_n$ incident to the colour $j$. Suppose for a contradiction that $N_j(u) \subseteq N_k(v)$. Note that $N_k(u)$ is disjoint from $N_j(u)$, and thus by \ref{itm:AC1} and \ref{itm:AC2}, we have
    $$\bigl|N_k(v)\bigr| \geq \bigl|N_j(u)\bigr| + \deg_k\bigl(u; N_k(v)\bigr) + 1 \geq n/2. $$
    However, this contradicts the fact that $|N_k(v)| < n/2$. Thus, there exists a $j$-coloured edge $uw$ with $w \in N_\ell(v)$ for some $\ell \neq k$, as desired.

    For \ref{itm:P2}, suppose $\mathcal{C}_v = \{a,b,c\}$ and let $uw \in N_a(v) \times N_b(v)$ be a $j$-coloured edge. By \ref{itm:AC1}, the edge $uw$ can form monochromatic triangles with at most
    \[\bigl|N_a(v)\bigr| - \deg_a\bigl(u; N_a(v)\bigr) - 1 \leq \bigl|N_a(v)\bigr| - \abund - 1 = \gamma_a - 1\] vertices of $N_a(v)$ and, by the same computation, at most $\gamma_{b} - 1$ vertices of $N_{b}(v)$. Observe that equation~\eqref{eqn:gamma_sum} implies $\gamma_a + \gamma_b + \gamma_c \leq \abund + 1$. Hence, $uw$ forms monochromatic triangles with at least $\abund - \gamma_a - \gamma_b + 2 \geq \gamma_c$ vertices of $N_c(v)$. These are all transversal with respect to $\{N_a(v), N_b(v), N_c(v)\}$, as required by~\ref{itm:P2}.

    Combining \ref{itm:P1} and \ref{itm:P2} shows that there is at least one $j$-coloured transversal triangle with respect to $\{N_i(v )\}_{i \in \mathcal{C}_v}$. Then, for each $k \in \mathcal{C}_v$, the existence of $\gamma_k$~vertices in $N_k(v)$ incident with the colour~$j$ follows by an application of \ref{itm:P2} to each of the three edges of this transversal triangle.
\end{proof}

The following lemma proves the existence of vertices incident with exactly three colours. Moreover, it shows that these vertices can be chosen such that they satisfy properties desirable for the construction of a rainbow $K_4$. Note that all the lemmas so far (including \cref{lem:CF37needsnospan}) do not require the assumption that the given $\abund$-abundant colouring uses at least six colours.   

\begin{lem}
    \label{lem:CF37needsnospan}
    Let $\chi$ be a $\abund$-abundant colouring of $K_n$ with colour set $\mathcal{C}$. If $|\mathcal{C}| \geq 5$, then there exists a vertex $v \in V(K_n)$ which is incident with exactly $3$ colours. Moreover, $v$ may be chosen so that there exists another vertex $u \in V(K_n)$ with $|\mathcal{C}_u \setminus \mathcal{C}_v| \geq 2$. 
\end{lem}

\begin{proof}
    Let $\chi$ be a $\abund$-abundant colouring of $K_n$ with colour set $\mathcal{C}$, and let $|\mathcal{C}| \geq 5$. To prove the first part of the lemma, suppose for a contradiction that $|\mathcal{C}_v| \leq 2$ for every vertex $v \in V(K_n)$. 

     We start by claiming that $\chi$ contains no rainbow triangle. Suppose otherwise; then by relabelling the colours if necessary, we may assume that $\chi$ contains a triangle $abc$ with $\chi(ab) = 1$, $\chi(bc) = 2$, and $\chi(ac) = 3$. Since $|\mathcal{C}| \geq 5$, there is a vertex $d \notin \{a,b,c\}$ incident with a colour not in $[3]$. As $|\mathcal{C}_v| \leq 2$ for each $v \in \{a,b,c\}$, the edges $da$, $db$, and $dc$ must use at least two distinct colours in $[3]$. This implies $|\mathcal{C}_d| \geq 3$, giving a contradiction.
    
    Thus, we may assume that $\chi$ contains no rainbow triangle. Now, let $x \in V(K_n)$ be some vertex incident with exactly two colours, say $\mathcal{C}_{x} = \{1, 2\}$. Note that such a vertex must exist, as otherwise the colouring $\chi$ would be monochromatic. Without loss of generality, we may assume that $|N_1(x)| \leq |N_2(x)|$. Then by \ref{itm:AC2}, we have that $n/4 < |N_1(x)| \leq n/2$ and $n/2 \leq |N_2(x)| < 3n/4$. 
    
    Observe, using \ref{itm:AC2}, that $\chi$ contains an $\ell$-coloured star on at least $\abund + 2$ vertices for each colour $\ell \in \mathcal{C}$. However, as there are no rainbow triangles, every edge between $N_1(x)$ and $N_2(x)$ uses either colour $1$ or $2$. This implies that any such $\ell$-coloured star with $\ell \notin \{1,2\}$ is either entirely contained in $N_1(x)$ or in $N_2(x)$. Together with \ref{itm:AC1} applied to the centre of the star in $N_1(x)$, the first of these options implies $|N_1(x)| \geq 2\abund + 2 > n/2$. This gives a contradiction, and so all these monochromatic stars must be contained in $N_2(x)$. 
    
    As $|\mathcal{C}| \geq 5$, there are three monochromatic stars of distinct colours not in $[2]$, each on at least $\abund + 2$ vertices and contained entirely within $N_2(x)$. Observe that these stars cannot all be pairwise vertex-disjoint, as otherwise \(|N_2(x)| \geq 3 \abund + 6 > 3n/4\), a contradiction. Hence, at least two of these stars intersect, and so any vertex~$u$ in their intersection must satisfy $|\mathcal{C}_u| \geq 3$. This proves the first part of the lemma.
    
    Now let~$v \in V(K_n)$ be a vertex incident with three colours, say $\mathcal{C}_v = [3]$. If there exists a vertex $u \in V(K_n)$ with $|\, \mathcal{C}_u \setminus \mathcal{C}_v| \geq 2$, then the second part of the lemma holds. Hence, we may assume that no such vertex exists. Let $k$ and $\ell$ be any two colours in $\mathcal{C} \setminus \mathcal{C}_v$. Then, by \cref{lem:pivot_basic_props}, there exist vertices $u_1, w_1 \in N_1(v)$ and $u_2, w_2 \in N_2(v)$ such that the edges $u_1u_2$ and $w_1w_2$  are coloured $k$ and $\ell$, respectively. Note that if $u_1 = w_1$ we have $\{\ell, k \} \subseteq \mathcal{C}_{u_1} \setminus \mathcal{C}_v$, a contradiction. Thus, $u_1 \neq w_1$ and, for the same reason, $u_2 \neq w_2$.
    
    Next, consider the edge $u_1w_2$. By the choice of the vertices $u_1$ and $w_2$, and by our assumption, it follows that the edge~$u_1w_2$ uses a colour in $[3]$. If $\chi(u_1w_2) = 3$, then we have $\mathcal{C}_{u_1} = \{1,3,k\}$ and $\mathcal{C}_{w_2} = \{ 2, 3, \ell\}$, thereby proving the lemma. If not, then without loss of generality we have $\chi(u_1w_2) = 1$ and so $\mathcal{C}_{w_2} = \{1,2,\ell\}$. Then, by \cref{lem:pivot_basic_props}, there exists a vertex in $ N_3(v)$ that is incident with colours $\{3, k\} \subseteq \mathcal{C} \setminus \mathcal{C}_{w_2}$, as desired.  
\end{proof}

\subsection{Finding a rainbow \texorpdfstring{$K_4$}{K4}}

Next, we use \cref{lem:CF37needsnospan} to prove that a $\abund$-abundant colouring~$\chi$ contains a rainbow $K_4$, under the assumption that $\chi$ uses at least six colours. This is done in \cref{lem:symmetric_four}. For the proof of \cref{lem:symmetric_four}, we first require a counting lemma (\cref{lem:technical_lemma} below) that relies on the abundance property of~$\chi$ applied to the edges of a particularly placed rainbow path~$P_4$.  

In order to accompany the proofs with figures, we switch notation for the remainder of this section, and use letters instead of numbers to denote colours in $\mathcal{C}$. Specifically, we use $R$, $G$, and $B$ to denote colours shown as red, green, and blue; and $D$, $W$, and $Y$ to denote colours represented by dashed, wavy, and solid black edges. We will often use the first three for the coloured neighbourhoods of a vertex $v$ with $|\mathcal{C}_v| = 3$, whereas the others will play the role of the colours in $\mathcal{C} \setminus \mathcal{C}_v$. The colours have been chosen to highlight the slightly varied, but inherently symmetric, roles played by the colours in $\mathcal{C}_v$ and $\mathcal{C}\setminus \mathcal{C}_v$ in the following proofs. Finally, for any set of colours $\{X,Y,Z\} \subseteq \mathcal{C}$, we use \define{$V_{XYZ}$} to denote the set of vertices $v \in V(K_n)$ for which $\mathcal{C}_v = \{X,Y,Z\}$.

\vspace{10pt}
\begin{lem}\label{lem:technical_lemma}
    Let $\chi$ be a $\abund$-abundant colouring of $K_n$ with colour set $\mathcal{C}$. Suppose $|\mathcal{C}| \geq 6$ and let $v\in V(K_n)$ be a vertex satisfying $\mathcal{C}_v = \{R,G,B\}$. Suppose that there exist vertices $u, u' \in N_R(v)$ and $w, w' \in N_B(v)$ such that the path $\mathcal{P} = u' w u w'$ has the following properties.
    \begin{itemize}
        \item  $\mathcal{P}$ forms a rainbow path using colours not in $\{R,G,B\}$.
        \item There are $t$ vertices in $N_G(v)$ that do not form monochromatic triangles with any edge of $\mathcal{P}$.
    \end{itemize}
    Then, there exist at least $\gamma_G(v) + t + 2$ vertices in $N_G(v)$, each of which forms a monochromatic triangle with both the edges~$u'w$ and~$uw'$ of $\mathcal{P}$.
\end{lem}
\begin{proof}
Set $\gamma_i \coloneqq \gamma_i(v)$ for each $i \in \{R,G,B\}$. Suppose that $\chi(u'w) = Y$, $\chi(wu) = W$, and $\chi(u w') = D$ (see \cref{fig:GYD}). For each $x \in \{u, u', w, w'\}$ and $c \in \{D, W, Y\}$, let $\sigma_c(x)$ denote the number of $c$-coloured edges incident with $x$ which lie entirely within the set $N_i(v)$ containing the vertex $x$. That is, if $x \in N_i(v)$, then $\sigma_c(x) = \deg_c\bigl(x; N_i(v)\bigr)$. Then by~\ref{itm:AC1}, we have the following inequalities. 
$$\sigma_Y(u') < \gamma_R\,; \;\; \sigma_Y(w) + \sigma_W(w) <  \gamma_B\,; \;\;  \sigma_W(u) + \sigma_D(u) < \gamma_R\,; \; \text{ and }\;\; \sigma_W(w') < \gamma_B.$$
Here, the second and third inequality use the fact that, for instance, $N_Y(w)$ and $N_W(w)$ are disjoint.

For each edge $e \in \{u'w, wu, uw'\}$, let $M_e$ denote the set of vertices in $N_G(v)$ that form monochromatic triangles with the edge $e$. Observe that, as $wu$ intersects both $u'w$ and $uw'$ while using a distinct colour from both of them, we have $M_{wu} \cap (M_{u'w} \cup M_{uw'}) = \emptyset$. Moreover, as $\chi(wu) = W$ and $M_{wu} \subseteq N_G(v)$ by definition, it follows from the $\abund$-abundance property that
    \(|M_{wu}| \geq \abund - \sigma_W(u) - \sigma_W(w)\). Recall from the assumptions of the lemma that there are $t$ vertices in $N_G(v)$ that do not belong to $M_{u'w} \cup M_{wu} \cup M_{uw'}$. Thus we have,
\begin{eqnarray*}    
|M_{u'w} \cup M_{uw'}| &\leq& |N_G(v)| - t - |M_{wu}| = \abund + \gamma_G - t - |M_{wu}| \\ &\leq& \gamma_G + \sigma_W(u) + \sigma_W(w)- t. 
\end{eqnarray*}
On the other hand, analogous to $M_{wu}$, we have $|M_{u'w}| \geq  \abund - \sigma_Y(u') - \sigma_Y(w)$ and $|M_{uw'}| \geq \abund - \sigma_D(u) - \sigma_D(w')$. The lemma now follows by combining all the above inequalities, to return the lower bound
    \begin{eqnarray*}
        \bigl|M_{u'w}&\cap&M_{uw'}\bigr|= \bigl|M_{u'w}\bigr| + \bigl|M_{uw'}\bigr| - \bigl|M_{u'w} \cup M_{uw'}\bigr|\\
        &\geq& 2 \abund + t - \gamma_G - \bigl(\sigma_Y(u') + \sigma_Y(w) + \sigma_W(w) + \sigma_W(u) + \sigma_D(u) +\sigma_D(w')\bigr)\\
        &\geq&  2 \abund + t - \gamma_G - 2 \gamma_B - 2\gamma_R + 4\\
        &=& 2\Bigl(\abund - \gamma_G - \gamma_B -\gamma_R + 2\Bigr) + \gamma_G + t\: \geq\: \gamma_G + t + 2,
    \end{eqnarray*}
    where in the last inequality we used $\gamma_G + \gamma_B + \gamma_R \leq \abund + 1$ by equation~\eqref{eqn:gamma_sum}.
\end{proof}

We now prove the existence of a rainbow $K_4$.
%under the assumption that $\chi$ uses at least six colours.

               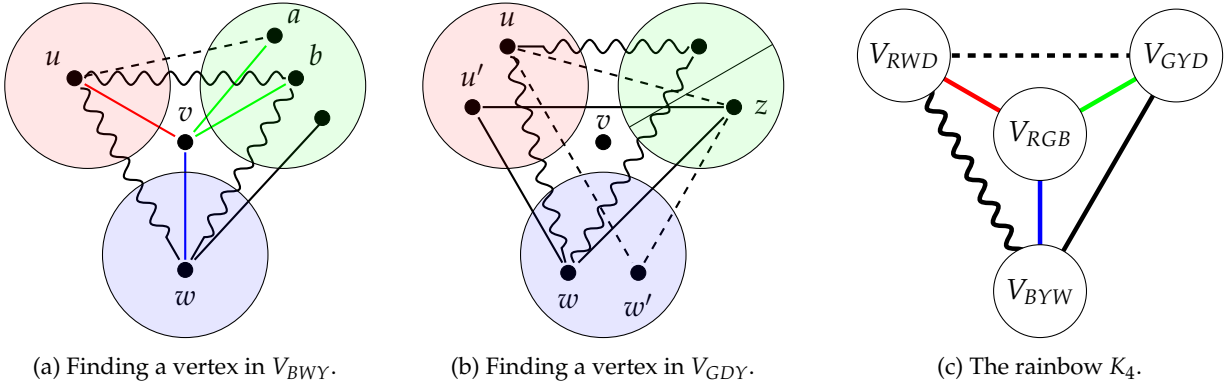
\begin{figure}[t]
    \centering
    \begin{subfigure}[b]{0.3\textwidth}
    \centering
    \begin{tikzpicture}[baseline]
        \node (x) at (150:1.7) {};
        \node (b) at (30:1.7) {};
        \node (c) at (270:1.7) {};
        \node (a) at (50:1.85) {};
        \node (o) at (0,0) {};
        \node (null) at (10:1.85) {};
        
        \node[above = 1.4mm] at (o) {$v$};
        \node[above left = 0.5mm] at (x) {$u$};
        \node[above right = 0.5mm] at (b) {$b$};
        \node[above right = 0.5mm] at (a) {$a$};
        \node[below = 1.4mm] at (c) {$w$};

        \draw[thick, red] (o) -- (x);
        \draw[thick, blue] (o) -- (c);
        \draw[thick, green] (o) -- (b);
        \draw[thick, green] (o) -- (a);
        \draw[thick, decorate,decoration={coil,aspect=0}] (x) -- (c);
        \draw[thick, decorate,decoration={coil,aspect=0}] (x) -- (b);
        \draw[thick, decorate,decoration={coil,aspect=0}] (b) -- (c);
        \draw[thick, black] (c) -- (10:1.9);
        \draw[thick, black, dashed] (x) -- (a);
        
        \node[draw] at (a) [circle,fill,inner sep=2pt]{};
        \node[draw] at (x) [circle,fill,inner sep=2pt]{};
        \node[draw] at (c) [circle,fill,inner sep=2pt]{};
        \node[draw] at (b) [circle,fill,inner sep=2pt]{};
        \node[draw] at (o) [circle,fill,inner sep=2pt]{};
        \node[draw] at (null) [circle,fill,inner sep=2pt]{};
        
        \filldraw[ fill = red, fill opacity = 0.1] (150:1.5) circle (1.1);
        \filldraw[ fill = green, fill opacity = 0.1] (30:1.5) circle (1.1);
        \filldraw[ fill = blue, fill opacity = 0.1] (270:1.5) circle (1.1);

        \end{tikzpicture}
        \caption{Finding a vertex in $\Vb$.}
        \label{fig:BYW}
        \end{subfigure}
        \hfill
        \begin{subfigure}[b]{0.30\textwidth}
        \centering
        \begin{tikzpicture}[grow = right, baseline]
        
        \node (x) at (135:1.8) {};
        \node (z) at (15:1.8) {};
        \node (c) at (255:1.8) {};
        \node (d) at (165:1.8) {};
        \node (w) at (45:1.8) {};
        \node (y) at (285:1.8) {};
        \node (o) at (0,0) {};
        
        \node[above = 0.4mm] at (o) {$v$};
        \node[above = 1.4mm] at (x) {$u$};
        \node[right = 1.4mm] at (z) {$z$};
        \node[below = 1.4mm] at (c) {$w$};
        \node[above = 1.4mm] at (d) {$u'$};
        \node[below = 1.4mm] at (y) {$w'$};

        %\draw[ red] (o) -- (x);
        %\draw[blue] (o) -- (y);
        %\draw[green] (o) -- (z);
        %\draw[red] (o) -- (d);
        %\draw[blue] (o) -- (c);
        %\draw[green] (o) -- (w);
        \draw[thick, decorate,decoration={coil,aspect=0}] (x) -- (c);
        \draw[thick, decorate,decoration={coil,aspect=0}] (c) -- (w);
        \draw[thick, decorate,decoration={coil,aspect=0}] (w) -- (x);
        \draw[thick, black] (c) -- (d);
        \draw[thick, black] (c) -- (z);
        \draw[thick, black] (d) -- (z);
        \draw[thick, black, dashed] (x) -- (z);
        \draw[thick, black, dashed] (y) -- (z);
        \draw[thick, black, dashed] (x) -- (y);
        \draw[black] (30:0.4) -- (30:2.6);
        
        \node[draw] at (x) [circle,fill,inner sep=2pt]{};
        \node[draw] at (y) [circle,fill,inner sep=2pt]{};
        \node[draw] at (z) [circle,fill,inner sep=2pt]{};
        \node[draw] at (w) [circle,fill,inner sep=2pt]{};
        \node[draw] at (c) [circle,fill,inner sep=2pt]{};
        \node[draw] at (d) [circle,fill,inner sep=2pt]{};
        \node[draw] at (o) [circle,fill,inner sep=2pt]{};
        
        \filldraw[ fill = red, fill opacity = 0.1] (150:1.5) circle (1.1);
        \filldraw[ fill = green, fill opacity = 0.1] (30:1.5) circle (1.1);
        \filldraw[ fill = blue, fill opacity = 0.1] (270:1.5) circle (1.1);

        \end{tikzpicture}
        \caption{Finding a vertex in $\Vg$.}
        \label{fig:GYD}
        \end{subfigure}
        \hfill
        \begin{subfigure}[b]{0.33\textwidth}
        \centering
        \begin{tikzpicture}[grow = right, baseline,scale = 0.95]

        \node (RDW) at (150:2.2) {};
        \node (GYD) at (30:2.2) {};
        \node (BYW) at (270:2.2) {};
        \node (RGB) at (0,0) {};
  
        \draw[line width = 0.6mm, decorate,decoration={coil,aspect=0}] (RDW) -- (BYW);
        \draw[line width = 0.6mm, red] (RDW) -- (RGB);
        \draw[line width = 0.6mm, blue] (BYW) -- (RGB);
        \draw[line width = 0.6mm, green] (GYD) -- (RGB);
        \draw[line width = 0.6mm, black] (GYD) -- (BYW);
        \draw[line width = 0.6mm, black, dashed] (GYD) -- (RDW);

        \draw[fill = white] (150:2.2) circle (0.65);
        \draw[fill = white] (30:2.2) circle (0.65);
        \draw[fill = white] (270:2.2) circle (0.65);
        \draw[fill = white] (0,0) circle (0.65);

        \node at (RGB) {$V_{RGB}$};
        \node at (RDW) {$V_{RWD}$};
        \node at (GYD) {$V_{GYD}$};
        \node at (BYW) {$V_{BYW}$};

        \end{tikzpicture}
        \caption{The rainbow $K_4$.}
        \label{fig:4setstruct}
        \end{subfigure}
    \caption{The structures involved in the proof of \cref{lem:symmetric_four}.}
    %\label{fig:6-col linkG structure}
\end{figure}

\begin{lem}
\label{lem:symmetric_four}
     Let $\chi$ be a $\abund$-abundant colouring of $K_n$ with colour set $\mathcal{C}$. If $|\mathcal{C}| \geq 6$, then the colouring~$\chi$ contains a rainbow $K_4$. 
\end{lem}
\begin{proof}
    Let $\chi$ be an $\abund$-abundant colouring of $K_n$ with colour set $\mathcal{C}$ having $|\mathcal{C}| \geq 6$. By \cref{lem:CF37needsnospan}, there exists a vertex $v \in V(K_n)$ incident with three colours in $\chi$, say $\mathcal{C}_v = \{R,G,B\}$. Moreover, the vertex $v$ may be chosen so that, for some vertex $u \in V(K_n)$, we have $|\mathcal{C}_u \setminus \mathcal{C}_v| \geq 2$. Let $D$, $Y$, and $W$ be three colours in $\mathcal{C}\setminus \{R,G,B\}$. Then, by possibly relabelling the colours, we may assume that the vertex~$u$ lies in $N_R(v)$ and that $\mathcal{C}_u = \{R,D,W\}$. For each $i \in \mathcal{C}_v$, recall that $\gamma_i \coloneqq \gamma_i(v)= |N_i(v)| - \abund$. Again by relabelling if necessary, we may assume that $|N_G(v)| \leq |N_B(v)|$, and thus $\gamma_G \leq \gamma_B$.
    
    Consider the vertex $u \in N_R(v)$. By \cref{lem:pivot_basic_props}, there exist distinct vertices $a,b \in N_G(v)$ such that the edges $ua$ and $ub$ are coloured $D$ and $W$, respectively (see \cref{fig:BYW} for reference). Let $M \subseteq V(K_n)$ be the set of vertices in $K_n$ that form a monochromatic triangle with either $ua$ or $ub$. Note that no vertex $m \in M$ forms a monochromatic triangle with both these edges, as the edge $um$ can only take one of these colours. Hence, we have $|M| \geq 2\abund$. 
    
    By an application of \ref{itm:AC1} to the vertex $u \in N_R(v)$, we have \[\bigl|M \cap N_R(v)\bigr| \leq \bigl|N_R(v)\bigr| - \deg_R\bigl(u; N_R(v)\bigr) - 1 \leq  \gamma_R-1.\] By a similar computation for the vertices $a, b \in N_G(v)$, and by using the assumption that $\gamma_G \leq \gamma_B$, it follows that $|M \cap N_G(v)| \leq 2\gamma_G - 2 \leq \gamma_G + \gamma_B -2$. Combining these inequalities with \eqref{eqn:gamma_sum} yields
    \begin{eqnarray*}
        |M \cap N_B(v)| &=& |M| - |M\cap N_R(v)| -|M\cap N_G(v)| \\ &\geq& 2 \abund - (\gamma_R + \gamma_G + \gamma_B) + 3 \geq \abund + 2.
    \end{eqnarray*}
    However, by \cref{lem:pivot_basic_props}, the colour $Y$ is incident with at least $\gamma_B$ vertices of $N_B(v)$. Hence, as $|N_B(v)| = \abund +\gamma_B$, there exists a vertex $w \in M \cap N_B(v)$ that is incident to the colour~$Y$. By possibly relabelling the colours $D$ and $W$ (and switching the labels of $a$ and $b$), we may assume that~$w$ is incident with the colours $\mathcal{C}_w = \{B, W, Y\}$.

    Now, consider the $W$-coloured edge $uw$. By an application of \cref{lem:pivot_basic_props} for the colours $D, Y \notin \mathcal{C}_v$, there exist vertices $u' \in N_B(v)$ and $w' \in N_R(v)$ such that the edges $u'w$ and $uw'$ are coloured $Y$ and $D$, respectively. Note that $u'wuw'$ is a rainbow $\{Y,W,D\}$-coloured path, alternating between the sets $N_R(v)$ and $N_B(v)$ (see \cref{fig:GYD} for reference). Applying \cref{lem:technical_lemma} to this rainbow path shows that, for some $t \geq 0$, there are at least $\gamma_G + t + 2 $ vertices of $N_G(v)$ that form monochromatic triangles with both $u'w$ and $uw'$. Then, any such vertex $z$ satisfies $\mathcal{C}_z = \{G, Y, D\}$. 
    
    Summarising, we have shown that there exists a set of four vertices $\{v, u, w, z\}$ in $K_n$, such that we have $v \in \Vberg$, $u \in \Vr$, $w \in \Vb$, and $z \in \Vg$. Recall that here, $V_{ijk}$ is the set of vertices incident with exactly the colours $i$, $j$, and $k$. The colours of all the edges between these vertices are forced, and in fact they must form a rainbow $K_4$, as shown in \cref{fig:4setstruct}. \end{proof}

\subsection{Proof of \texorpdfstring{\cref{lem:UB_colours}}{Lemma \ref{lem:UB_colours}}}

We are now ready to prove~\cref{lem:UB_colours}. Without loss of generality, the vertices of the rainbow $K_4$ returned by \cref{lem:symmetric_four} belong to the vertex sets $\Vberg$, $\Vr$, $\Vg$, and $\Vb$, which themselves form a blow-up of a rainbow $K_4$ (see \cref{fig:4setstruct}). As outlined above, we obtain bounds on the size of this blow-up by applying the abundance property of $\chi$ to the edges of various substructures, such as a rainbow triangle and certain rainbow $P_4$s. Finally, we show that these structures impose mutually exclusive constraints on the size of the blow-up, thereby leading to a contradiction.   

\begin{proof}[Proof of \cref{lem:UB_colours}]
    Let $\chi$ be a $\abund$-abundant colouring of $K_n$ with colour set $\mathcal{C}$. Suppose for a contradiction that $|\mathcal{C}| \geq 6$. By \cref{lem:symmetric_four}, the colouring~$\chi$ contains a rainbow~$K_4$. Furthermore, as $|\mathcal{C}_v| \leq 3$ for any vertex $v \in V(K_n)$, we may assume by possibly relabelling the colours that this rainbow $K_4$ is formed by choosing a vertex from each of the four sets $\Vberg$, $\Vr$, $\Vg$, and $\Vb$. Recall that $V_{ijk}$ is the set of vertices incident with exactly the colours $i$, $j$, and $k$. In fact, the colours of all edges between these four sets are forced and they must form a (possibly unbalanced) blow-up of a rainbow $K_4$ (see \cref{fig:4setstruct} for reference).   
    
    Set $\Lambda \coloneqq |\Vberg| + |\Vr| + |\Vg| + |\Vb|$. Then, we have the following claim.

    \begin{claim}
    \label{clm:rgbineq}
    For each index $A \in \{{\rm RGB}, {\rm RDW}, {\rm GDY}, {\rm BWY}\}$, we have
    $$2|V_A| <  \Lambda - n/4 .$$
    \end{claim}
    \begin{claimproof}
    By the symmetry of the colouring, it suffices to prove the claim for $A = {\rm RGB}$. Consider the vertices $v \in \Vberg$, $a \in \Vr$, $b \in \Vg$, and $c \in \Vb$, that form a rainbow $K_4$ in the colouring~$\chi$. Set $\gamma_i \coloneqq \gamma_i(v)$ for all $i \in \mathcal{C}_v$. For each edge $e \in E(K_n)$, let $M_e$ be the set of vertices that form monochromatic triangles with $e$ that are transversal with respect to $\{N_R(v), N_G(v), N_B(v)\}$. So, for instance, we have $M_{ab} \subseteq N_B(v)$,  $M_{bc} \subseteq N_R(v)$, and $M_{ac} \subseteq N_G(v)$. By the property~\ref{itm:P2}, for any edge $xy \in N_R(v) \times N_G(v)$ having a colour $\chi(xy) \notin \{R,G,B\}$, we have \(|M_{xy}| \geq \gamma_B \geq 1\). Similarly, we have that $|M_{xy}| \geq \gamma_G \geq 1$ if $xy \in N_R(v) \times N_B(v)$ and $|M_{xy}| \geq \gamma_R \geq 1$ if $xy \in N_G(v) \times N_B(v)$. Hence, let $a'$, $b'$, and $c'$ be arbitrarily chosen vertices from the sets $M_{bc}, M_{ac}$, and $M_{ab}$, respectively. See \cref{fig:6-col linkG structure} for an illustration.        % Then, it suffices to show that \[|\Vr| + |\Vg| + |\Vb| \geq |\Vberg| + n - 1 - 3\lfloor n/4 \rfloor\] 
    
    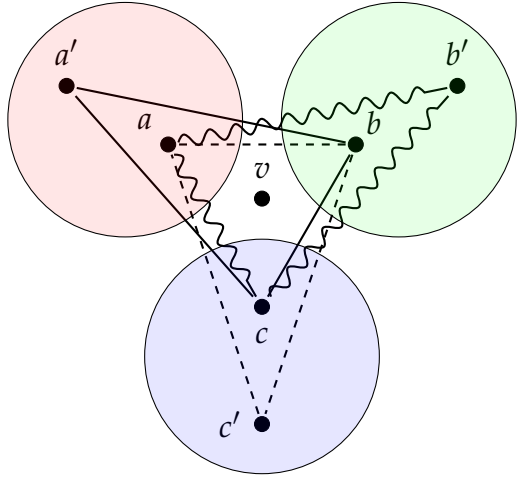
\begin{figure}[t]
\centering
\begin{tikzpicture}[scale = 1.2]

\node (RDW) at (150:1.2) {};
\node (GYD) at (30:1.2) {};
\node (BYW) at (270:1.2) {};
\node (RY) at (150:2.5) {};
\node (GW) at (30:2.5) {};
\node (BD) at (270:2.5) {};
\node (O) at (0,0) {};

\node[above = 1.4mm] at (O) {\large $v$};
\node[above left = 1mm] at (RDW) {\large $a$};
\node[above = 1.4mm] at (RY) {\large $a'$};
\node[above right] at (GYD) {\large $b$};
\node[above = 1.4mm] at (GW) {\large $b'$};
\node[below = 1.4mm] at (BYW) {\large $c$};
\node[left = 1.4mm] at (BD) {\large $c'$};
  
\draw[thick, decorate,decoration={coil,aspect=0}] (RDW) -- (BYW);
\draw[thick, decorate,decoration={coil,aspect=0}] (RDW) -- (GW);
\draw[thick, decorate,decoration={coil,aspect=0}] (BYW) -- (GW);
%\draw[thick, red] (RDW) -- (RY);
%\draw[thick, blue] (BYW) -- (BD);
%\draw[thick, green] (GYD) -- (GW);
\draw[thick, black] (GYD) -- (BYW);
\draw[thick, black] (GYD) -- (RY);
\draw[thick, black] (BYW) -- (RY);
\draw[thick, black, dashed] (GYD) -- (RDW);
\draw[thick, black, dashed] (GYD) -- (BD);
\draw[thick, black, dashed] (BD) -- (RDW);

\node[draw] at (RDW) [circle,fill,inner sep=2pt]{};
\node[draw] at (RY) [circle,fill,inner sep=2pt]{};
\node[draw] at (GYD) [circle,fill,inner sep=2pt]{};
\node[draw] at (GW) [circle,fill,inner sep=2pt]{};
\node[draw] at (BYW) [circle,fill,inner sep=2pt]{};
\node[draw] at (BD) [circle,fill,inner sep=2pt]{};
\node[draw] at (O) [circle,fill,inner sep=2pt]{};

\filldraw[ fill = red, fill opacity = 0.1] (150:1.75) circle (1.3);
\filldraw[ fill = green, fill opacity = 0.1] (30:1.75) circle (1.3);
\filldraw[ fill = blue, fill opacity = 0.1] (270:1.75) circle (1.3);

\end{tikzpicture}

    \caption{The structure involved in the proof of \cref{lem:UB_colours}.}
    \label{fig:6-col linkG structure}
\end{figure}
    
    First, we obtain lower bounds on $|\Vb|$, $|\Vr|$, and $|\Vg|$. To bound $|\Vb|$ from below, consider the rainbow $\{Y, D, W\}$-coloured path $a' b a b'$, which alternates between the sets~$N_R(v)$ and~$N_G(v)$. Let $L_B$ be the set of vertices in $N_B(v)$ that do not form monochromatic triangles with any of $a'b$, $ba$, or $ab'$. Thus, an application of \cref{lem:technical_lemma} to $\mathcal{P} = a'bab'$ and with $t = |L_B|$ returns a set of at least $\gamma_B + |L_B| + 2$ vertices in $N_B(v)$ that form monochromatic triangles with both $a'b$ and $ab'$. All these vertices must belong to $\Vb$, and so $|\Vb| \geq \gamma_B + |L_B| + 2$.

    Next, we repeat the above argument with respect to the sets $\Vr$ and $\Vg$. This tells us that $|\Vr| \geq \gamma_R + |L_R| + 2$, where $L_R$ is the set of vertices in $N_R(v)$ not forming monochromatic triangles with any of $b'c, cb,$ or $bc'$. Similarly, $|\Vg| \geq \gamma_G + |L_G| + 2$, where $L_G$ consists of the vertices in $N_G(v)$ not forming monochromatic triangles with $a'c, ca,$ or $ac'$.

    Finally, observe that we have $\Vberg \cap N_R(v) \subseteq L_R$, as any vertex in $N_R(v) \setminus L_R$ forms a monochromatic triangle with an edge among $b'c, cb, $ and $bc'$ and is thus incident with a colour in $\{D,W,Y\}$. Analogously, we have $\Vberg \cap N_G(v) \subseteq L_G$ and $\Vberg \cap N_B(v) \subseteq L_B$. Combining these observations yields $\Vberg \setminus \{v\}  \subseteq L_R \cup L_G \cup L_B$. Thus, we have
\begin{eqnarray*}
    |\Vr| + |\Vg| + |\Vb| &\geq& \gamma_R + \gamma_G+ \gamma_B + |L_R| + |L_B| + |L_G| + 6 \\ &\geq& \abund -2 + (|\Vberg|-1) + 6 \\ &=& \abund + |\Vberg| + 3 \\ 
    &>& n/4 + |\Vberg|,
\end{eqnarray*}
where the second inequality uses the above observation along with equation~\eqref{eqn:gamma_sum}. The claim now follows by adding $|\Vberg|$ on both sides of the above inequality and rearranging its terms.
\end{claimproof}

Similarly to the proof of \cref{clm:rgbineq}, choose vertices $v \in \Vberg$, $a \in \Vr$, $b \in \Vg$, and $c \in \Vb$ (as in \cref{fig:6-col linkG structure}). Let $\gamma_i \coloneqq \gamma_i(v)$ and for each edge $e$, write $M_e$ to denote the set of vertices forming monochromatic triangles with $e$ that are transversal with respect to $\{N_R(v), N_G(v), N_B(v)\}$. 

Observe that $abc$ is a rainbow triangle and so no vertex forms monochromatic triangles with two or more of its edges. This implies that there are at least $3\abund$ vertices forming a monochromatic triangle with an edge of $abc$. On the other hand, the number of vertices in $N_R(v)$ forming monochromatic triangles with some edge of $abc$ that are \emph{not} transversal (relative to $\{N_i(v)\}_{i \in \{R,G,B\}}$) is at most $\deg_W\bigl(a; N_R(v)\bigr) + \deg_D\bigl(a; N_R(v)\bigr) \leq \gamma_R - 1$ by \ref{itm:AC1}. Similarly, there are at most $\gamma_G - 1$ such vertices in $N_G(v)$ and at most $\gamma_B -1$ in $N_B(v)$. Therefore, we have 
\[\bigl|M_{ab} \cup M_{bc} \cup M_{ca}\bigr| \geq 3\abund - \gamma_R - \gamma_G - \gamma_B + 3 \geq 2 \abund + 2 \geq n/2,\]
where the second inequality follows from equation~\eqref{eqn:gamma_sum}. 

Note that every vertex of $M_{ab}$ is incident with the colours $B$ and $D$; every vertex of $M_{bc}$ with $R$ and $Y$; and every vertex of $M_{ca}$ with $G$ and $W$. Hence, $M_{ab}$, $M_{bc}$, and $M_{ca}$ are disjoint from each of the sets $\Vberg$, $\Vr$, $\Vg$, and $\Vb$. Together with the above inequality, this implies that $\Lambda = |\Vberg| + |\Vr| + |\Vg| + |\Vb| \leq n/2$. However, this contradicts \cref{clm:rgbineq}, as adding the four inequalities returned by the claim for each index~$A \in \{{\rm RGB}, {\rm RWD}, {\rm GYD}, {\rm BYW}\}$ yields $2\Lambda < 4\Lambda - n$, which rearranges to $\Lambda > n/2$. Therefore, by contradiction, any $\abund$-abundant edge-colouring of $K_n$ uses at most $5$ colours. \end{proof}

\section{Characterizing small configurations}\label{sec:rconfigs}

In this section we prove \cref{lem:config_classif}, which captures the key properties of $r$-configurations for $r\in \{4,5,6\}$ needed in our proof of \cref{thm:main_colouring}.

\begin{proof}[Proof of \cref{lem:config_classif}\ref{lem:config_classif:1}] Let $\chi: E(K_4^{(3)}) \to \mathbb{N}$ be a $4$-configuration. Recall that by \cref{defn:rconfig}, the colouring $\chi$ is an edge colouring of $K_4^{(3)}$ with no spanning colour. As any pair of edges in $K_4^{(3)}$ span all $4$ vertices, the colouring $\chi$ must be rainbow, as required.
\end{proof}

\begin{proof}[Proof of \cref{lem:config_classif}\ref{lem:config_classif:2}]
Let $\chi: E(K_5^{(3)}) \to \mathbb{N}$ be a $5$-configuration on the vertex set $S = \{v_i\}_{i\in [5]}$. For each colour $j \in \mathbb{N}$, let $H_j \coloneqq H_j^{\chi}$ denote the $j$-coloured subgraph in $\chi$. By \cref{defn:rconfig}, the induced colouring $\chi[S \setminus \{v_i\}]$ has a spanning colour that is not incident with the vertex $v_i$, for each $i \in [5]$. By possibly relabelling the colours, we may assume that this colour is $i$. So, for each colour~$i \in [5]$, the subgraph~$H_i$ spans $4$ vertices, and hence we have $e(H_i) \geq 2$. However as $K_5^{(3)}$ has ten edges, we have that $\sum_{i \in \mathbb{N}} e(H_i) = 10$. Combining these facts, we see that $e(H_i) = 2$ for each $i \in [5]$ and $e(H_i) = 0$ for $i\ge 6$, thereby proving the statement. 
\end{proof}

To prove \cref{lem:config_classif}\ref{lem:config_classif:3}, we will show that in any $6$-configuration, every colour has to appear at least $3$ and at most $5$ times.
We will then show that if such a configuration $\chi$ contains a monochromatic copy of $K_4^{(3)}$, then either $\chi$ contains a rainbow $K_4^{(3)}$, contradicting \ref{config:2}, or $\chi$ contains a pair of vertices lying in four distinctly coloured edges.

\begin{obs}\label{obs:r-1verticesineachcolour}
    If an $r$-configuration uses precisely $r$ colours, then every colour spans exactly $r-1$ vertices.
\end{obs}
\begin{proof}
Let $V \coloneqq V(K_r^{(3)})$. By property \ref{config:2}, for each $v \in V$, $\chi[V \setminus \{v\}]$ has a spanning colour $c_v \in [r]$. By property \ref{config:1}, $c_v \neq c_{v'}$ for any distinct $v, v' \in V$. This shows that there are at least $|V| = r$ distinct colours spanning exactly $r-1$ vertices, thereby proving the observation.   
\end{proof}

\begin{lem}\label{lem:atleast3edges}
If a $6$-configuration uses $6$ colours, then every colour class uses at least $3$ edges.
\end{lem}

\begin{proof}
    Let $\chi : E(K_6^{(3)}) \to [6]$ be a $6$-configuration using precisely $6$ colours. By \cref{obs:r-1verticesineachcolour}, every colour spans exactly $5$ vertices. Then we may label the vertices in $V(K_6^{(3)})$ as $v_1, \dots, v_6$ so that, for each $i \in [6]$, $v_i$ is the (unique) vertex not incident with any $i$-coloured edges.
    
    Note that each colour in $[6]$ uses at least two edges since it spans $5$ vertices. Assume for a contradiction that one of the colours, say colour $6$, only uses two edges. By relabelling the colours and vertices if necessary, we may assume that the $6$-coloured edges are $v_1v_2v_3$ and $v_1v_4v_5$. The triple $v_1v_2v_4$ cannot use a colour in $\{1,2,4,6\}$ and so, by relabelling $2 \leftrightarrow 4$ and $3 \leftrightarrow 5$ if necessary, we may assume that $\chi(v_1v_2v_4) = 3$. By property \ref{config:2}, $v_1v_2v_3v_4$ does not form a rainbow $K_4^{(3)}$. However, $\chi(v_1 v_3 v_4), \chi(v_2v_3v_4) \notin \{3,6\}$ and so $\chi(v_1 v_3v_4) = \chi(v_2v_3v_4)$. This forces $\chi(v_1v_3v_4) = \chi(v_2v_3v_4) = 5$. 

    Note that $\chi(v_1v_3v_5), \chi(v_3 v_4 v_5) \notin \{5,6\}$. As the quadruple $v_1 v_3v_4 v_5$ also does not induce a rainbow $K_4^{(3)}$, we must then have $\chi(v_1v_3v_5) = \chi(v_3v_4v_5) = 2$. Similarly, the quadruple $v_2v_3v_4v_5$ not being rainbow, together with $\chi(v_3v_4v_5) = 2$ and $\chi(v_2v_3v_4) = 5$, implies $\chi(v_2v_3v_5) = \chi(v_2v_4v_5) = 1$.

    Finally, consider the quadruple $v_1v_2v_3 v_5$, whose edges satisfy $\chi(v_1v_2v_3) = 6, \chi(v_1v_3v_5) = 2,$ and $ \chi(v_2v_3v_5) = 1$. Observe that $\chi(v_1v_2v_5) \notin \{1,2,6\}$, and so $v_1v_2v_3v_5$ induces a rainbow $K_4^{(3)}$, contradicting \ref{config:2}.
\end{proof}

\begin{proof}[Proof of \cref{lem:config_classif}\ref{lem:config_classif:3}] Let $\chi: E(K_6^{(3)}) \to [6]$ be a $6$-configuration using precisely $6$ colours. For each $i \in [6]$, let $H_i \coloneqq H_i^{\chi}$ be its $i$th colour class. Let $V \coloneqq V(K_6 ^{(3)})$. By \cref{obs:r-1verticesineachcolour}, we may label the vertices of $V$ as $v_1, \dots, v_6$ so that $v_i$ is the unique vertex not incident with colour $i$.

By relabelling the colours if necessary, we may assume that $e(H_1) \leq \dots \leq e(H_6)$.
By \cref{lem:atleast3edges}, we have $e(H_i) \geq 3$ for each $i \in [6]$.
Since $\sum_{i \in [6]} e(H_i)=e(K_6^{(3)}) = \binom{6}{3} = 20$, we therefore also have $e(H_i) \leq 5$ for each $i \in [6]$, as required by the first part of the lemma. 

To show the second part, let us assume for a contradiction that 
\begin{enumerate}[label = (\roman*)]
    \item\label{nopairin4colours} $\chi$ has no pair of vertices contained in edges of $4$ distinct colours, and 
    \item\label{monoK43} $\chi$ contains a monochromatic $K_4^{(3)}$.
\end{enumerate}
Note that the monochromatic $K_4^{(3)}$ has $4$ edges but only spans $4$ vertices. As the colour class it is contained in spans $5$ vertices, it must contain at least $5$ edges. This shows that $e(H_6)=5$.
It follows that $e(H_i)=3$ for each $i\in [5]$ since $e(K_6^{(3)}) = 20$.
Moreover, the monochromatic $K_4^{(3)}$ is a subgraph of $H_6$ since it has four edges.
We may assume, using the symmetry of the $K_4^{(3)}$ and relabelling the colours and vertices if necessary, that $H_6$ consists precisely of the four triples contained in the quadruple $v_1v_2v_3v_4$, together with the triple $v_3v_4v_5$.

For each $i \in [6]$, let $\phi_i$ be the number of pairs $ab \in \binom{V}{2}$ that are contained in at least two $i$-coloured edges. Note that each pair $ab \in \binom{V}{2}$ is contained in precisely four edges, and these four edges use at most three distinct colours by \ref{nopairin4colours}. Thus, 
\[\sum_{i \in [6]} \phi_i \geq \binom{6}{2} = 15.\]
Now, for each $i \in [6]$, let $\Phi_i$ be the number of quadruples $abcd \in \binom{V}{4}$ such that the induced subcolouring $\chi[abcd]$ uses at least two $i$-coloured edges. As $\chi$ contains no rainbow $K_4^{(3)}$ by \ref{config:2}, any quadruple contains at least two edges of the same colour. Thus, 
\[\sum_{i \in [6]} \Phi_i \geq \binom{6}{4} = 15.\]
Inspecting the $3$-graph $H_6$, we see that $\phi_6 = 6$ as the pairs of vertices contained in more than one $6$-coloured edge are those included in $v_1v_2v_3v_4$. Moreover, the quadruples containing more than one $6$-coloured edge are $v_1v_2v_3v_4, v_1v_3v_4v_5,$ and $v_2v_3v_4v_5$, and so $\Phi_6 = 3$.

Let $i\in [5]$.
Then $e(H_{i}) = 3$, hence $\Phi_i\leq \binom{3}{2}=3$ since any two distinct triples belong to at most one quadruple.
Relabelling if necessary, we may assume $\Phi_1<\cdots <\Phi_5$.
Observe that $\sum_{i \in [5]} \Phi_i \geq 15 - 3 = 12$.
It follows that $\Phi_{4} = \Phi_{5}=3$.
Moreover, if $\Phi_i=3$ for some $i$, then $H_{i}$ is isomorphic to the $3$-graph on vertex set $\{a, b,c,d,e\}$ and edges $abc,abd,abe$, in which case $\phi_i=1$.
Therefore $\phi_4 = \phi_5=1$.
So $\sum_{i \in [3]} \phi_i \geq 15 - 6 - 1 - 1 = 7$, meaning there is some $j\in [3]$ such that $\phi_{j}\geq 3$. Since two distinct edges cannot intersect at more than one pair of vertices, the three edges in $H_j$ must pairwise intersect in two vertices.
This means that $H_{j}$ spans four vertices, thereby giving a contradiction.
\end{proof}

\section{Concluding remarks}\label{sec:conclud}

In this paper, we determined the minimum codegree threshold for the appearance of a spanning tight component in a $4$-uniform hypergraph. We achieved this by reducing the problem to finding a spanning colour in an edge-colouring of $K_n^{(3)}$ in which each edge lies in many monochromatic~$K_4^{(3)}$s. To solve this problem, we bounded the number of colours appearing in the colouring, which allowed us to find a very small $r$-configuration. Using the structural properties of this $r$-configuration, together with the existence of vertices forming monochromatic $K_4^{(3)}$s with many of its edges, we derived a contradiction. Our work suggests several interesting questions, discussed below. 

\paragraph{Higher uniformities.} The most prominent open problem is to attack \cref{conj:main} for higher uniformities; that is, to determine the minimum codegree threshold for spanning tight components in $k$-graphs with $k \geq 5$. Our reformulation extends to this setting, where it concerns edge-colourings of $K_n^{(k-1)}$ in which each edge lies in $n/k$ monochromatic $K_k^{(k-1)}$s. Alhough in principle our strategy for \cref{thm:main_colouring} might extend too, several obstacles would need to be overcome.

To bound the total number of colours, we used a `colour-merging' procedure to ensure that each vertex is incident with all but one colour, and then proved that every vertex sees at most $5$ distinct colours. Moreover, various steps in the proof used the fact that every pair of vertices lies in edges of at most $3$ distinct colours. Hence, answering the following question seems a natural first step towards \cref{conj:main}. 

\begin{prob}
    Let $\chi$ be an edge-colouring of $K_n^{(k-1)}$ in which each edge lies in $n/k$ monochromatic $K_k^{(k-1)}$s. For each $1 \leq d \leq k-2$, what is the maximum, over all $d$-tuples $S \subseteq V(K_n^{(k-1)})$, of the number of edges of distinct colours containing $S$? 
\end{prob}

Equivalently, this asks for an upper bound on the number of colours appearing in the coloured link graphs of $d$-tuples, which are defined in the obvious way. For $d = k-2$, an upper bound of $k-1$ follows from the same argument used for \cref{obs:basic_facts}\ref{basic_obs:fact2}. When $d = k-3$, the coloured link graph of a $d$-tuple is a $n/k$-abundant edge-colouring of $K_{n-k+3}$. This leads to the following question, which may be of independent interest. 

\begin{prob}
    For each $t \geq 1$, determine the infimum $\alpha = \alpha(t) \geq0$ with the following property. For all $\alpha' > \alpha$ and $n$ sufficiently large, every $\alpha'n$-abundant colouring of $K_n$ uses at most $t$ colours.
\end{prob}

It is not hard to see that $\alpha(1) = 1/2$ and $\alpha(2) = \alpha(3) = 1/3$. Note that \cref{lem:UB_colours} implies that $\alpha(5) \leq 1/4$. Somewhat surprisingly, this bound is tight. In particular, for all $m \geq 1$ there exists a construction of a $(8m+1)$-abundant edge-colouring of the complete graph on $32m+7$ vertices using $6$ colours---see \cref{fig:extremal} and its caption. We make no conjecture about the value of $\alpha(t)$ for~$r \geq 6$.

Finally, $r$-configurations generalize naturally to higher uniformities as vertex-minimal structures with no spanning colour. In our proof, the bound on the number of colours allowed us to control the size of the $r$-configuration, which could then be characterised precisely because it was small. It would be interesting to characterise $r$-configurations without size restrictions, as this might allow one to prove \cref{conj:main} by contradiction without precise control on the number of colours. Note that if every edge lies in $n/k$ monochromatic $K_k^{(k-1)}$s, a rough bound of $O(k^{k})$ on the total number of colours is easy to obtain by counting the edges in each colour. It is unclear, however, whether this helps in proving the conjecture.

\begin{figure}[t]
\begin{center}
\begin{tikzpicture}
            \coordinate (1) at (90:2.5cm);
            \coordinate (2) at (90+ 360/7:2.5cm);
            \coordinate (3) at (90+ 2*360/7:2.5cm);
            \coordinate (4) at (90+ 3*360/7:2.5cm);
            \coordinate (5) at (90+ 4*360/7:2.5cm);
            \coordinate (6) at (90+ 5*360/7:2.5cm);
            \coordinate (7) at (90+ 6*360/7:2.5cm);
            
            \node[above=4mm] at (1) { $$8m+1$$};
            \node[left=4mm] at (2) {$$4m+1$$};
            \node[left=4mm] at (3) { $$4m+1$$};
            \node[below=4mm] at (4) {$$4m+1$$};
            \node[below=4mm] at (5) {$$4m+1$$};
            \node[right=4mm] at (6) {$$4m+1$$};
            \node[right=4mm] at (7) {$$4m+1$$};
            
            \draw[very thick, blue] (1) -- (2);
            \draw[very thick, blue] (1) -- (3);
            \draw[very thick, red] (1) -- (4);
            \draw[very thick, red] (1) -- (5);
            \draw[very thick, red] (1) -- (6);
            \draw[very thick, red] (1) -- (7);
            \draw[very thick, blue] (2) -- (3);
            \draw[very thick, black] (2) -- (4);
            \draw[very thick, yellow] (2) -- (5);
            \draw[very thick, black] (2) -- (6);
            \draw[very thick, yellow] (2) -- (7);
            \draw[very thick, green] (3) -- (4);
            \draw[very thick, green] (3) -- (5);
            \draw[very thick, cyan] (3) -- (6);
            \draw[very thick, cyan] (3) -- (7);
            \draw[very thick, green] (5) -- (4);
            \draw[very thick, black] (6) -- (4);
            \draw[very thick, red] (7) -- (4);
            \draw[very thick, red] (5) -- (6);
            \draw[very thick, yellow] (5) -- (7);
            \draw[very thick, cyan ] (6) -- (7);
            
            \foreach \n in {1,2,3,4,5,6,7}{
            \draw[ fill = white] (\n) circle(0.4);
            }

            \tkzDrawSector[R,fill=red, fill opacity = 0.6](1,0.4)(90,-90);
            \tkzDrawSector[R,fill=blue, fill opacity = 0.6](1,0.4)(-90,-270);
            \tkzDrawSector[R,fill=black, fill opacity = 0.6](2,0.4)(90,-90);
            \tkzDrawSector[R,fill=yellow, fill opacity = 0.6](2,0.4)(-90,-270);
            \tkzDrawSector[R,fill=cyan, fill opacity = 0.6](3,0.4)(90,-90);
            \tkzDrawSector[R,fill=green, fill opacity = 0.6](3,0.4)(-90,-270);
            \tkzDrawSector[R,fill=green,fill opacity = 0.6](4,0.4)(90,-90);
            \tkzDrawSector[R,fill=black, fill opacity = 0.6](4,0.4)(-90,-270);
            \tkzDrawSector[R,fill=green, fill opacity = 0.6](5,0.4)(90,-90);
            \tkzDrawSector[R,fill=yellow, fill opacity = 0.6](5,0.4)(-90,-270);
            \tkzDrawSector[R,fill=cyan, fill opacity = 0.6](6,0.4)(90,-90);
            \tkzDrawSector[R,fill=black, fill opacity = 0.6](6,0.4)(-90,-270);
            \tkzDrawSector[R,fill=yellow, fill opacity = 0.6](7,0.4)(90,-90);
            \tkzDrawSector[R,fill=cyan, fill opacity = 0.6](7,0.4)(-90,-270);
            
\end{tikzpicture}
    \caption{An example of a $(8m+1)$-abundant colouring of $K_{32m+7}$ using $6$ colours. The vertex set is partitioned into seven clusters, one of size $8m+1$ and the others of size $4m+1$. Inside each cluster of size $4m+1$, the two colours depicted in the figure are used, each inducing a $2m$-regular graph spanning the cluster; the same applies for the larger cluster, except that in this case the colour classes are $4m$-regular. The edges between the clusters behave according to the diagram.}
    \label{fig:extremal}
\end{center}
\end{figure}
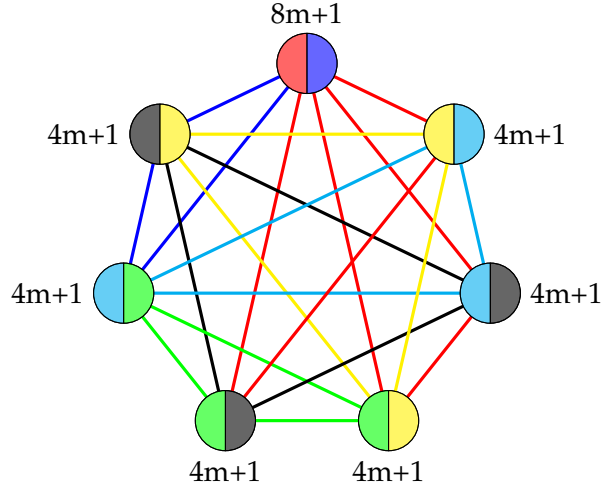

\paragraph{Spheres and large components.} Recall that Georgakopoulos, Haslegrave, Montgomery, and Narayanan \cite{spanningsurfaces} conjectured that $n/k$ is in not only the minimum codegree threshold in $k$-graphs for spanning tight components, but also for spanning $(k-1)$-spheres. They proved this asymptotically for $k=3$. It seems plausible that our methods and results could be used to make progress on this conjecture for $k=4$, even in the case of minimum codegree $n/4 + o(n)$, where one expects the spanning component to be robustly connected.

In a different direction, Georgakopoulos, Haslegrave, and Montgomery \cite{forcinglargetight} considered the problem of determining, in $3$-graphs, which minimum codegree forces a tight component of any prescribed size. They uncovered an interesting phenomenon: the resulting threshold function exhibits (possibly infinitely) many discontinuities. A similar behaviour was observed by Allen, B\"ottcher, and Hladk\'y~\cite{turanposa} when studying which minimum degree in a graph forces a square of a path/cycle of any given size (see their paper for the definition). This was later extended by Hng~\cite{powerspaths} to higher powers of paths/cycles. A connection between these problems is to be expected, since the triangles inside the square of a path form a tightly connected $3$-graph. It would be natural to investigate whether analogous behaviour occurs for tight components in $4$-graphs and, relatedly, for powers of tight paths/cycles in $3$-graphs.

\paragraph{Acknowledgment.} The work leading up to this paper began at the Staycation Workshop hosted by LSE in April 2025. We are grateful to Leo Versteegen for organizing the event.

%%%%%%%%%%%%%%%%%%%%%%%%%%%%%%%%%%%%

{
\fontsize{11pt}{12pt}
\selectfont
	
\hypersetup{linkcolor={red!70!black}}
\setlength{\parskip}{2pt plus 0.3ex minus 0.3ex}

\newcommand{\etalchar}[1]{$^{#1}$}
}

\end{document}